\documentclass[10pt,twoside,a4paper]{amsart}

\usepackage{amsfonts,amsmath,amsthm}
\usepackage{hyperref}

\newtheorem{thm}{Theorem}[section]

\newtheorem{lem}[thm]{Lemma}
\newtheorem{pro}[thm]{Proposition}
\newtheorem{cor}[thm]{Corollary}

\theoremstyle{remark}
\newtheorem{remark}[thm]{Remark}

\theoremstyle{ass}

\newtheorem{exa}[thm]{Example}

\newcommand{\R}{\mathbb{R}}
\newcommand{\Rd}{\mathbb{R}^{d}}

\newcommand{\D}{\nabla_{A}}

\newcommand{\F}{\mathcal{F}}
\newcommand{\Dr}{\mathcal{D}_{r}}

\numberwithin{equation}{section}

\title[Magnetic Hamiltonians]{The forward problem for the electromagnetic Helmholtz equation with critical singularities}

\author[J.A. Barcel\'o, L. Vega, M. Zubeldia]{Juan Antonio Barcel\'o$^{1}$, Luis Vega$^{2}$ and Miren Zubeldia$^{3}$\\
\\
$^{1}$ \emph{ETSI de Caminos, Universidad Polit\'ecnica de Madrid, 28040, Madrid, Spain}\\
\\
$^{2, 3}$ \emph{Departamento de Matem\'aticas, Universidad del Pa\'is Vasco, Apartado 644, 48080, Bilbao, Spain.}
}

\date{\today}

\thanks{First and second authors were supported by the Spanish Grant MTM2011-28198. Third author was partially supported by the Spanish grant FPU AP2007-02659 of the MEC, by the Spanish Grant MTM2007-62186 and by the ERC Advanced Grant-Mathematical foundations (ERC-AG-PE1) of the European Research Council}



\begin{document}

\maketitle

\begin{abstract}
We study the forward problem of the magnetic Schr\"odinger operator with potentials that have a strong singularity at the origin. We obtain new resolvent estimates and give some applications on the spectral measure and  on the solutions of the associated evolution problem.  
\end{abstract}

\section{Introduction}

Let us consider the magnetic Schr\"odinger operator
\begin{equation}
H_{A} = \D^{2} + V
\end{equation}
with
$$
\D := \nabla + iA.
$$
Here $A: \Rd \to \Rd$ is the magnetic vector potential and $V:\Rd \to \R$ is the electric scalar potential. The magnetic potential $A$ describes the interaction of a free particle with an external magnetic field. The magnetic field that corresponds to a magnetic potential $A$ is given by the $d \times d$ anti-symmetric matrix defined by
\begin{equation}\label{rotacional}
B=(DA)-(DA)^{t}, \quad B_{kj}=\left(\frac{\partial A_{k}}{\partial x_{j}}- \frac{\partial A_{j}}{\partial x_{k}}\right) \quad k,j=1,\ldots,d.
\end{equation}
In dimension $d=3$, $B$ is identified by the vector field $curl \, A$ via the vector product
\begin{equation}
B v = curl \, A \times v, \quad \forall v \in \R^{3}.
\end{equation}
Of particular interest is the tangential part of $B$ given as
\begin{equation}\label{tangen0}
B_{\tau}(x) = \frac{x}{|x|}B(x), \quad \quad \quad \quad (B_{\tau})_{j} = \sum_{k=1}^{d}\frac{x_{k}}{|x|} B_{kj}
\end{equation}
 Observe that $B_{\tau} \cdot x = 0$ for any $d\geq 2$ and  therefore $B_{\tau}$ is a tangential vector field in any dimension.

In this paper we are interested in the study of some new resolvent estimates for the electromagnetic Helmholtz equation
\begin{equation}\label{magneticequation5}
(\nabla + iA)^{2}u + Vu + \lambda u \pm i\varepsilon u = f \quad \quad \quad \varepsilon >0,
\end{equation}
with critical singularities on the potentials. More precisely we are able to consider potentials  such that $V$ and $|x|^{2}|B_{\tau}|^{2}$  can have singularities at the origin as $\frac{(d-2)^2}{ 4|x|^{2}}$. In addition, our results are true for singular magnetic potentials $A$ such that $B_{\tau}=0$. Relevant examples of magnetic potentials that we can consider are the following. 
\begin{exa}\label{ejemploB1}
Take
\begin{equation}\label{Apoten}
A= \frac{C}{x^{2} + y^{2} + z^{2}}(-y, x, 0) = \frac{C}{x^{2}+y^{2}+z^{2}}(x,y,x) \times (0,0,1).
\end{equation}
One can easily check that
\begin{equation}\nabla \cdot A =0, \quad \quad B=-2C\frac{z}{(x^{2}+y^{2}+z^{2})^{2}}(x,y,z), \quad \quad B_{\tau}=0.
\end{equation}
\end{exa}
\begin{exa}\label{ejemploB2}
Another (more singular) example is the well known Aharonov-Bohm (A-B) potential
\begin{equation}
A=C\left(\frac{-y}{x^{2}+y^{2}}, \frac{x}{x^{2}+y^{2}}, 0 \right) = \frac{1}{x^{2}+y^{2}}(x,y,z)\times (0,0,1).
\end{equation}
Here we have $B=C(0,0,\delta)$, where $\delta$ is the Dirac's delta function and we obtain $B_{\tau}=0$. Note that we can consider this example in the two dimensional case. \end{exa}

\begin{exa}\label{ejemploB3}
Potentials that satisfy 
\begin{equation}
x \cdot A(x) =0 \quad \quad \text{and} \quad \quad \lambda A(\lambda x) = A(x),
\end{equation}
for all $\lambda>0$.
Notice that the previous two examples satisfy both properties.
In fact it is observed in \cite{FVV}, and easy to check, that these two properties imply that $B_{\tau}=0$.
\end{exa}

It is interesting to recall that if $A$ is regular then one can always assume that
\begin{equation}
x \cdot A(x) =0,
\end{equation}
which  is known as the Cr\"omstrom gauge. This is done by considering $A-\nabla m$ instead of $A$, with  the function $m(x)$ given by
\begin{equation}\label{m}
m(x)=\sum_{j=1}^{d} x_{j} \int_{0}^{1} A_{j}(tx) dt, \quad \quad (x\in\Rd).
\end{equation}
Observe that for the examples given above $m$ can not be defined. In addition, if the Cr\"omstrom gauge is used then 
\begin{equation}
A_k(x)=\sum_{j=1}^{d}  \int_{0}^{1}tx_{j} B_{jk}(tx) dt, \quad \quad x\in\Rd,\qquad k=1,\ldots,d,
\end{equation}
and therefore $A=0$ if $B_\tau=0$.

Let us introduce the framework in which we work in the sequel. From now on we assume that $H_{A}$ is a self-adjoint operator in $L^{2}(\Rd)$ with form domain $D(H_{A})=: H^{1}_{A}(\Rd) = \{\phi \in L^{2}(\Rd) : \int |\D \phi|^{2}  <\infty\}$. In particular this is achieved if
\begin{equation}\label{assV4}
\int |V||u|^{2} \leq \nu\int |\nabla_{A} u|^{2} +C_\nu \int | u|^{2}\quad \quad 0 <  \nu < 1
\end{equation}
and
\begin{equation}\label{localintegrability}
A_{j}\in L^{2}_{loc}.
\end{equation}
See the standard references \cite{CFKS}, \cite{LS} for details. From the self-adjointness of the operator $H_{A}$ we deduce the existence of solution of the equation (\ref{magneticequation5}) in the Hilbert space $H_{A}^{1}(\Rd)$.

Note that the A-B potential given in example \ref{ejemploB2} does not satisfy \eqref{localintegrability}. For this particular example we will consider the Friedrichs extension.

In order to state the first result of this paper we need to impose the following assumptions:
\begin{itemize}
\item [(\bf{H1})] 
\begin{equation}\notag
\int V|u|^{2} \leq \nu\int |\nabla_{A} u|^{2} \quad \quad 0 <  \nu < 1.
\end{equation}
\item [(\bf{H2})] 
\begin{align}
\int (\partial_{r}(rV))_{-} |u|^{2} < A_{V} \int |\nabla_{A} u|^{2},\notag
\end{align}
\begin{align}
\left(\int |x|^{2}|B_{\tau}|^{2}|u|^{2}\right)^{1/2} < A_{B}\left(\int |\nabla_{A} u|^{2} \right)^{1/2},\notag
\end{align}
\end{itemize}
with 
\begin{equation}\label{peque}
A_{V} + 2A_{B} < 1,
\end{equation}
where $r= |x|$, $\partial_{r} = \frac{x}{|x|}\cdot \nabla $, and  $(x)_- =-min(x,0)$.

Observe that if we replace the integrals on the right-hand side of the above hypotheses by $\int |\nabla u|^{2}$, then by the diamagnetic inequality
\begin{equation}\label{diamagnetic}
|\nabla |f|| \leq |\D f|,
\end{equation}
(see \cite{LL}) we obtain (H1) and (H2).




\begin{exa}\label{ex}
For $d\geq 3$, let us take
\begin{equation}
(\partial_{r}(rV))_{-} \leq \frac{\nu_{1}}{|x|^{2}},\quad \quad |B_{\tau}|\leq \frac{\nu_{2}}{|x|^{2}},
\end{equation}
where
$$
\frac{4\nu_{1}}{(d-2)^{2}} + \frac{2\nu_{2}}{(d-2)} < 1.
$$
More concretely we can take, 
\begin{equation}
V= \frac{\nu_{1}}{|x|^{2}}\quad \textrm{with} \quad 0<\nu_{1}< \frac{(d-2)^{2}}{4}.
\end{equation}
Then, one can easily check that $(H1)$ and $(H2)$ hold.
\end{exa}

\begin{exa}\label{exam}
One can also work with an electric potential $V$ of Coulomb type and a long range magnetic potential $A$. In fact, if we take
\begin{equation}
V(x) = \frac{V_{\infty}\left(\frac{x}{|x|}\right)}{|x|^\alpha} \quad \textrm{with} \quad V_{\infty} < 0, \quad 1\leq \alpha\leq 2,
\end{equation}
and
\begin{equation}\notag
|B_{\tau}|\leq \left\{ \begin{array}{ll}
\frac{d-2}{2|x|^{2}} & \textrm{if $d\geq3$}\\
\\
\frac{\sqrt{C_{H_{2}}}}{|x|^{2}} & \textrm{if $d=2$},
\end{array} \right.
\end{equation}
where the constant $C_{H_{2}}>0$ is related to the two dimensional magnetic Hardy inequality given by Laptev and Weidl in \cite{LW}. See (\ref{Hardydimension2}) below. Then $\partial_{r}(rV) \geq 0$ and our assumptions are satisfied.
\end{exa}

\vspace{0.2cm}
Now we are ready to state the first theorem. 

\begin{thm}\label{Theorem1}
Let $d\geq 3$, $\varepsilon > 0$, $f$ such that $\Vert |\cdot|f\Vert_{L^{2}}<\infty$ and assume that  {\bf(H1)} and {\bf(H2)} hold. Then, there exists $C>0$ independent of $\lambda, \varepsilon,$ and $f$ such that any solution $u\in H^{1}_{A}(\Rd)$ of the electromagnetic Helmholtz equation (\ref{magneticequation5}) satisfies 
\begin{itemize}
\item[(i)] for $0<\varepsilon < \lambda$
\begin{equation}\label{alpha0}
\int |\D(e^{\mp i\lambda^{1/2}|x|}u)|^{2} \leq C\int |x|^{2}|f|^{2};
\end{equation}
\item[(iii)] for $\lambda \leq \varepsilon$
\begin{equation}
\int |\D u|^{2} \leq C\int |x|^{2}|f|^{2}.
\end{equation} 
\end{itemize}
For $d=2$ the same results hold if moreover the magnetic potential $A$ is continuous on $\R^{2} \backslash \{0\}$ and $curl \, A \in L^{1}_{loc}(\R^{2}\backslash \{0\})$.
\end{thm}


By the magnetic Hardy inequality
\begin{equation}\label{A.111}
\int \frac{|f|^{2}}{|x|^{2}} \leq \frac{4}{(d-2)^{2}} \int |\D f|^{2},
\end{equation}
which is true for any $f\in H^{1}_{A}(\Rd)$ with $d\geq 3$, we deduce the uniform resolvent estimate
\begin{align}\label{BP}
\int \frac{|u|^{2}}{|x|^{2}}  \leq C\int |x|^{2}|f|^{2},
\end{align}
for all $d\geq 3$ and for any $\lambda \in \R$. This extends the work  by Burq, Planchon, Stalker and Tahvilder-Zadeh \cite{BPST1}, \cite{BPST2} where they consider the purely electric case, i.e. $A=0$. Moreover, although our assumptions on the potentials are similar to their hypotheses, it turns out that conditions (H1) and (H2) are weaker than the ones required in their works. Indeed, the bounds that they impose on $V$ are for each sphere and uniform in the radius of the sphere. See assumptions (A2), (A3) in \cite{BPST2}.

In the two dimensional case, we use the following magnetic Hardy inequality
\begin{equation}\label{Hardydimension2}
\int_{\R^{2}} \frac{|u|^{2}}{|x|^{2}} \leq C_{H_{2}}\int_{\R^{2}} |\D u |^{2} \quad \quad u \in C^{\infty}_{0}(\R^{2}\backslash \{0\})
\end{equation}
given by Laptev and Weidl in \cite{LW} and true under the extra restrictions on the magnetic potential $A$ made in Theorem \ref{Theorem1} for $d=2$. Here the constant $C_{H_{2}} >0$ is not explicit and is related to the magnetic flux. See \cite{LW} and also \cite{FFT} for more details. As a consequence, (\ref{BP}) still holds for $d=2$.

From the above inequalities and using very simple duality arguments, in section \ref{section6} we will obtain weighted Sobolev uniform estimates in the spirit of the work of Kenig, Ruiz, and Sogge \cite{KRS}. Although our results do not recover those in \cite{KRS} the class of weights we obtain seems to be new even in the constant coefficient case $A=0$ and $V=0$, see remark 6.3 for more details. 

These weighted Sobolev inequalities can also be understood as smoothing estimates in the sense of Kato and Yajima \cite{KaYa} and therefore they have implications for the corresponding  Schr\"odinger evolution equation. Note however, that magnetic hamiltonians involve first order perturbations of the free laplacian. As a consequence estimates involving the gradient of the $u$ itself, and not just on $e^{- i\lambda^{1/2}|x|}u$ as those given in Theorem \ref{Theorem1}, are necessary if one wants to obtain Strichartz estimates from inequalities of the resolvent as done by Fanelli and Vega \cite{FV} and by D'Ancona, Fanelli, Vega and Visciglia \cite{DFVV}. 

Moreover, and as we have already said, we can consider electromagnetic potentials that are long range. In fact, Strichartz estimates for the example \ref{exam} are false as proved by Goldberg, Vega and Visciglia 
if $A=0$ and $\max V_{\infty}=0$ \cite{GVV} (see also  the work by Fanelli and Garcia for examples with $A\neq0$ \cite{FG}).  The extra property $\max V_{\infty}=0$ implies the existence of special solutions that are almost wave guides and that rule out the possibility of proving the standard Strichartz estimates. These special solutions can be seen as a bad behavior of the resolvent in a neighborhood of  the zero energy. In fact, if $\max V_{\infty}<0$ then there exists an extra inequality as proved by Barcel\'o, Ruiz, Vega and Vilela \cite{BRVV} that implies a better dispersion for the perturbed Hamiltonian than for the free case.
In most of the works devoted to dispersive estimates for large  potentials
assumptions involving the absence of either zero resonances or eigenvalues are usually imposed, see for example the work of Journ\'e, Soffer and Sogge \cite{JSS}, of Rodnianski and Schlag \cite{RS}, and of Erdogan, Goldberg and Schlag \cite{EGS} among others.

Very recently, Fanelli, Felli, Fontelos and Primo \cite{FFFP} study the dispersive property of the Schr\"odinger equation with singular electromagnetic potentials. Furthermore, Mochizuki \cite{M3} proves the estimate (\ref{BP}) assuming the following smallness conditions on the potentials
\begin{equation}
\max \{|B(x)|, |V(x)|\} \leq \frac{\varepsilon_{0}}{|x|^{2}} \quad \quad \textrm{in} \quad \quad \Rd
\end{equation} 
where $0<\varepsilon_{0} < \frac{1}{4\sqrt{2}}$ ($d=3$) or $< \left(\frac{(d-1)(d-3)}{8}\right)^{\frac{1}{2}}$ $(d\geq 4)$. In our case, the smallness is given by (\ref{peque}) and the constant is determined by the one that appears in the standard Hardy inequality, as we can see in example \ref{ex}.

We emphasize that under the assumptions of Theorem \ref{Theorem1}  it may be concluded a uniqueness result for the equation 
\begin{equation}\label{resjuan}
(\nabla + iA)^{2} u + V u + \lambda u = f.
\end{equation}
As a consequence, following the work of Ikebe and Saito \cite{IS} and  of Zubeldia \cite{Z} we deduce the limiting absorption principle for the Helmholtz equation with potentials $V$ and $A$ that can have sharp singularities at the origin. See subsection \ref{LAPchap4} below. 

A problem closely related to the validity of the limiting absorption principle is the existence of the so called far field pattern 
which is defined by  \begin{equation}
g_{\lambda}(\omega) = \lambda^{-\frac{(d-3)}{4}}\lim_{r\to \infty} r^{\frac{d-1}{2}}e^{-i\lambda^{1/2}r} u(r\omega) \quad \quad \text{in} \quad \quad L^{2}(S^{d-1}),
\end{equation}
where $\omega = \frac{x}{|x|}$ and 
$$u:=
R(\lambda + i0)f = \lim_{\varepsilon \to 0^{+}} R(\lambda + i\varepsilon)f ,
$$ is the outgoing solution of  \eqref{resjuan}. Here
 $R(z)= (H_{A} + z)^{-1}$ with $z\in \mathbb{C}$ denotes the resolvent operator of $H_{A}$ and the limit is in $(H^{1}_{A})_{loc}(\Rd)$  \footnote{ A similar expression can be written  for the so called incoming solution when the negative sign in \eqref{resjuan} is considered}.  The scattering amplitudes are the data used in inverse scattering problems, although the measurable quantities are the scattering cross-sections that we shall define as
\begin{equation}\label{cross} 
\mathcal G_{\lambda}(\omega) = \lambda^{-\frac{(d-3)}{2}}\lim_{r\to \infty} r^{d-1}|u(r\omega)|^2 \quad \quad \omega\in S^{d-1} \quad \text{in} \quad L^{1}(S^{d-1}).
\end{equation}

Regarding the existence of the far field pattern for the magnetic case we have to mention \cite{Ki1}, \cite{Ki2}, \cite{Ku1}-\cite{Ku5}, \cite{S1}-\cite{S3}, as well as the classical work by Agmon and H\"ormander \cite{A}, \cite{AH} and \cite{H}, vol II and IV, that include short and long range perturbations of the free laplacian. In the  long range case a modification of the radial phase $e^{- i\lambda^{1/2}|x|}u$ is necessary and a $\mathcal C^2$ regularity condition for the coefficients is required. Note that in the case of magnetic hamiltonians, and due to the gauge invariance, the natural assumptions are on the magnetic field $B$ and not directly on the magnetic potential $A$. In fact, in the examples \ref{ejemploB3}
the angular part of the magnetic potential $A$ can be very rough as in the particular case of example \ref {ejemploB2}. 

We should also mention the paper by Iwatsuka \cite{Iw}, where the author also considers long range magnetic potentials and introduces a modification of the phase using the function $m$ given in \eqref{m}.  In that work, $V$ is assumed to be short range and $A_{j}(x) \in C^{2}(\Rd)$ such that  $|B_{jk}(x)| \leq C_{0}(1+|x|)^{-3/2-\mu}$, for some $C_{0}, \mu >0$. Some modifications of the arguments used in that paper will allow us to obtain the existence and uniqueness of the cross section under the same assumptions as the ones imposed in Theorem \ref{Theorem1}. Indeed, we do not have to make any  modification of the radial phase even though we consider long range perturbations. As it is well known for the classical laplacian and the Coulomb electric potential a logarithmic correction is necessary in the phase, so that there is no change in the cross section. Thus there is no contradiction with our results. The precise statement is given in Theorem \ref{azkenx}.

Nevertheless, we have to pay a price for the generality of the assumptions of Theorem \ref{azkenx}. This is in the sense that the limit is taken as in \eqref{cross}. We prove  that the limit is a measure such that the total variation has the expected property and can be used to obtain the spectral representation of $H_A$. This is done in section \ref{section5}, where estimates for the spectral measure are also obtained.

At this point the question of whether $\mathcal G_{\lambda}$ is in $L^1(S^{d-1})$ naturally arises and the answer, that is stated in Theorem \ref{cross-section}, is one of the important consequences of the second main result of this paper. The result is related to a sharp Sommerfeld radiation condition for solutions $u$ of the equation (\ref{resjuan}). In order to prove it we need to require some other conditions on the potentials.  We make the following assumption.

\vspace{0.2cm}
\begin{itemize}
\item[(\bf{H3})] 
\begin{equation}\notag
|B_{\tau}| + |V| \leq \left\{ \begin{array}{ll}
\frac{C}{|x|^{2-\alpha}} & \textrm{if $|x|\leq 1$, $d=3$}\\
\frac{c}{|x|^{2}} & \textrm{if $|x|\leq 1$, $d>3$}\\
\frac{C}{|x|^{3+\alpha}} & \textrm{if $|x|\geq 1$}.
\end{array} \right.
\end{equation}
\end{itemize}
for some $c, C>0$, $\alpha >0$.

Let us introduce some notation. For $u:\Rd \to \mathbb{C}$, we define the norms
\begin{equation}\notag
|||u|||_{R_{0}} := \sup_{R>R_{0}} \left(\frac{1}{R}\int_{|x|\leq R} |u(x)|^{2}\right)^{1/2}
\end{equation}
and 
\begin{equation}\notag
N_{R_{0}}(u) := \sum_{j > J}\left(2^{j+1}\int_{C(j)} |u|^{2} \right)^{1/2} + \left(R_{0}\int_{|x|\leq R_{0}} |u|^{2} \right)^{1/2}
\end{equation}
where $J$ is such that $2^{J-1} < R_{0} < 2^{J}$ and $C(j)=\{x\in\Rd : 2^{j}\leq |x|\leq 2^{j+1}\}$. The norms $|||u|||_{1}$ and $N_{1}(u)$ are known as Agmon-H\"ormander norms. We drop the index $R_{0}$ if $R_{0}=0$. Note that
\begin{equation}\notag
\left| \int uv \right|  \leq |||u|||_{R_{0}}N_{R_{0}}(v) \quad \text{ for all \,} R_0\geq 0.
\end{equation}

\vspace{0.5cm}
We can now state the second result.

\begin{thm}\label{TheoremSRC}
For $d\geq3$, let $\lambda_{0} >0$ and $c>0$ small enough. Let $f$ such that $\Vert |\cdot|^{3/2}f\Vert_{L^{2}}<\infty$, $N_{1}(f)<\infty$ and assume that {\bf(H3)} holds. Then for any $\lambda \geq \lambda_{0}$
and $u=R(\lambda + i0)f$ we have
\begin{align}\label{keySRC}
\sup_{R\geq 1}\, R\int_{|x|\geq R} |\nabla_{A}(e^{-i\sqrt{\lambda}|x|}u)|^{2} \leq C\left[\int |x|^{3}|f|^{2} + (N_{1}(f))^{2}\right],
\end{align}
where $C=C(\lambda_{0})>0$.
\end{thm}

The radiation condition (\ref{keySRC}) extends the Sommerfeld condition given by
\begin{equation}\label{somalpha}
\int_{|x|\geq 1} |x|^{\alpha} |\D (e^{-i\lambda^{1/2}|x|}u)|^{2} < \infty
\end{equation}
for any $0<\alpha < 1$ and proved in \cite{Z}, which generalizes the one given by Ikebe and Saito \cite{IS}.  In order to get (\ref{keySRC}), we first need to control the Agmon-H\"ormander norm of the solution $u$. In fact, the smallness of the constant $c>0$ is only necessary for deducing the a-priori estimate
\begin{equation}\label{aprichapter41}
\lambda |||u|||_{1}^{2} \leq C(N_{1}(f))^{2}.
\end{equation}

It is interesting to observe that the estimate  \eqref{somalpha} is not enough to conclude that the cross section belongs to $L^1(S^{d-1})$ while from \eqref{keySRC} and \eqref{aprichapter41}
one easily concludes that $r_n^{\frac{d-1}{2}} u(r_n\omega)$ belongs to $H^1(S^{d-1})$ for a sequence of radii $r_n$ and with a uniform bound. The strong convergence of a subsequence follows from Rellich's compactness theorem. Therefore it is clear that any norm interpolated between \eqref{keySRC} and \eqref{aprichapter41} does the job. This suggests that quite likely the decay assumed in ({H3}) at infinity can be weakened to obtain the desired compactness result.

We also give another consequence of  Theorem \ref{TheoremSRC}. It is about the growth  of $\int_{R/2<|x|<R} |\nabla |u|^2|\,dx$ with $u$ a solution of (\ref{resjuan}) and $R>0$. We obtain a uniform bound  under some decay assumptions of the forcing term $f$. From this an estimate of $\|u\|_{L^{\frac{2d}{d-1}}(|x|\leq R)}$ with a logarithmic growth in $R$ easily follows. We obtain the right dependence in powers of $\lambda$ for  $\lambda>1$. The precise result is given in Theorem \ref{theoremgradient}.

It should be observed that while the estimates given in Theorem \ref{Theorem1} are uniform in $\lambda \in \R$, those of Theorem \ref{TheoremSRC} are uniform for $\lambda\geq\lambda_0>0$. We have decided to do so for simplicity. A crucial step in the proof of Theorem \ref{TheoremSRC} is to obtain a bound for the  term
\begin{equation}\label{ls}
\inf_{R/2<|x|<R}\int_{|x|=R} |u|^2 \,d\sigma_R,\qquad R>0.
\end{equation}
In order to prove a uniform estimate in $\lambda$ for this term  we would need assumptions similar to (H1) and (H2) but slightly different. In fact, we believe that  a uniform estimate in $\lambda$ and for all $R>0$ of \eqref{ls} does not hold for  the critical constant of Hardy's inequality $\frac{(d-2)^2}{4}$ but for  $\frac{(d-3)(d-1)}{4}$ that is the one obtained by Barcel\'o, Ruiz, and Vega \cite{BRV},  Fanelli and Vega \cite {FV} and Fanelli \cite{F}. See also Remark 5.4 below in connection with a recent work by Rodnianski and Tao \cite{RT}
where estimates valid for all $\lambda$ are obtained.
 In addition, it is to be expected that the decay of the electric potential $V$ required on assumption (H3) can be relaxed modifying the phase of the Sommerfeld radiation condition using the solution of the eikonal equation $|\nabla K|^{2} = \lambda + V$, as done by Saito \cite{S} and Zubeldia \cite{Z2}. We propose this as an interesting question which will be studied elsewhere. In any case this procedure will still use a compactness argument so that the dependence on $\lambda$ can not be quantified.
\vspace{0.5cm}

The rest of the paper is organized as follows. The next section provides the key identity that will be used in order to prove the main results. We emphasize that this identity holds in all dimensions. In section \ref{section2} we deal with the proof of the uniform resolvent estimates that has been stated in Theorem \ref{Theorem1}. Section \ref{section3} establishes the sharp Sommerfeld radiation condition (\ref{keySRC}), proving Theorem \ref{TheoremSRC}. Then we pass to present some applications of the main results. Section \ref{section5} is devoted to the study of the forward problem for the equation (\ref{resjuan}) including the existence and uniqueness of the cross section and its applications to the spectral resolution of the operator $H_A$. Finally  section \ref{section6} contains the weighted estimates for the resolvent and its applications to the evolution problem and the inequalities of $|\nabla |u|^2|$ that we have already mentioned.

\section*{Acknowledgements} 
The authors would like to express their gratitude to Luca Fanelli for many useful comments on the paper. The third author is also grateful to Ari Laptev for inviting her to participate in the Magnetic Fields program at the Mittag-Leffler Institute in October 2012 and for providing valuable information related to the dimension two.

\section{The key equality}

In this section, we consider the electromagnetic Helmholtz equation
\begin{equation}\label{magneticequation}
(\nabla + iA(x))^{2} u + V(x)u + \lambda u \pm i\varepsilon u = f(x), \quad x\in\Rd
\end{equation}
where $\lambda \in \R$, $\varepsilon > 0$ and we provide the key equality that will be used in the proofs of the main results. 

To this end, we first derive some integral identities based on the standard technique of Morawetz multipliers. This was introduced in \cite{Mo} for the Klein Gordon equation and then used in several other contexts as in dispersive equations, kinetic equations, Helmholtz equation, etc. We should also mention the papers \cite{IS},\cite{PV1}, and \cite{F}.

In order to carry out the integration by parts argument below, we need some regularity in the solution $u$. In general, it is enough to know that $u\in H^{1}_{A}(\Rd)$. Moreover, since we are including singularities in our potentials, it is necessary to put some restrictions on them to check that the contributions of these terms make sense. To this end, apart from the condition (\ref{assV4}), the following assumptions on $V$ and $B_{\tau}$ will be needed throughout the paper.
\begin{align}
& \int _{|x|<1}|r\partial_{r}V| |u|^{2} \leq C\int |\nabla u|^{2}, \label{assV1}\\
& \int _{|x|<1}|x|^{2}|B_{\tau}|^{2}|u|^{2} \leq C\int |\nabla u|^{2},\label{assB}
\end{align}
for some $C>0$. 

Before stating the integral identities, we need some notation. We denote the radial derivative and the tangential component of the gradient by 
\begin{equation}
\nabla_{A}^{r}u = \frac{x}{|x|}\cdot \nabla_{A}u, \quad \quad |\nabla_{A}^{\bot}u|^{2}=|\nabla_{A}u|^{2}-|\nabla_{A}^{r}u|^{2},
\end{equation}
respectively. We drop the index $A$ if $A=0$, and in this case we write the radial derivative as $\partial_{r}$ and the tangential one as $\nabla^{\bot}$. In addition, the following a-priori estimates are also necessary.

\begin{lem}\label{lemaapriori}
Let $\varphi\in C^{\infty}$ a real-valued radial function so that there exist $C>0$, $k_{0}\geq 0$ where $|\varphi^{(k_{0}+1)}|\leq C$. Then, for a suitable $f$, the solution of the Helmholtz equation (\ref{magneticequation}) satisfies
\begin{align}\label{a+++}
\varepsilon \int \varphi^{(k_{0})} |u|^{2} \leq \int |\varphi^{(k_{0}+1)}| |\D^{r}u | |u| + \int |f| |\varphi^{(k_{0})}||u|
\end{align}
\begin{align}\label{b+++}
\int \varphi^{(k_{0})}|\D u|^{2} & \leq \int (\lambda + |V|)\varphi^{(k_{0})}|u|^{2} + \int |\varphi^{(k_{0}+1)}||\D^{r}u| |u| + \int |f||\varphi^{(k_{0})}||u|.
\end{align}
\end{lem}

\begin{proof}
We just need to multiply the equation (\ref{magneticequation}) by $\varphi^{(k_{0})}\bar{u}$ and integrate it over $\Rd$. Then, the imaginary part gives (\ref{a+++}) and (\ref{b+++}) follows by taking the real part.
\end{proof}

In particular, it follows that
\begin{align}\label{a+++0}
\varepsilon \int  |u|^{2} \leq \int |f| |u|
\end{align}
\begin{align}\label{b+++0}
\int |\D u|^{2} & \leq \int (\lambda + V)|u|^{2}  + \int |f|||u|.
\end{align}

\begin{remark}\label{2+}
Note that under appropriate assumption on the potential $V$, if there exist $C>0$, $k_{0} \geq 0$ such that $|\varphi^{(k_{0}+1)}|\leq C$, by induction on $k$ we get 
\begin{equation}
\int \varphi^{(k)} (|u|^{2} + |\D u|^{2}) < +\infty \quad \quad \forall k \leq k_{0},
\end{equation}
for $u\in H^{1}_{A}(\Rd)$ and suitable $f$.
\end{remark}

Now we are ready to formulate the identities that will be used to prove the key equality. See \cite{F} for the proofs.

\begin{lem}\label{appendix1} Let $\varphi: \Rd \to \R$ be regular enough. Then, the solution $u\in H^{1}_{A}(\Rd)$ of the Helmholtz equation (\ref{magneticequation}) satisfies
\begin{align}\label{(4.11)}
& \int  \varphi\lambda |u|^{2} - \int \varphi |\D u|^{2}  + \int \varphi V|u|^{2}- \Re \int \nabla \varphi \cdot \D u \bar{u}=  \Re\int  \varphi f\bar{u},
\end{align}
\begin{equation}
\pm \varepsilon \int \varphi |u|^{2} \mp \Im \int \nabla \varphi\cdot \D u\bar{u} = \pm \Im\int \varphi f\bar{u}.\label{(4.2)}
\end{equation}
\end{lem}

\begin{lem}\label{appendix2}
Let $\psi: \mathbb{R}^n \longmapsto \mathbb{R}$ be radial and regular enough. Then, any solution $u\in H^{1}_{A}(\Rd)$ of the equation (\ref{magneticequation}) satisfies

\begin{align}\label{(4.3)}
&\int \D u\cdot D^{2} \psi \cdot \overline{\D u} +\Re \frac{1}{2}\int \nabla(\Delta \psi)\cdot \D u \bar{u}\pm \varepsilon \Im\int \nabla \psi\cdot \overline{\D u}u\\
& - \Im  \int \psi' B_{\tau} \cdot \D u\bar{u} +\frac{1}{2}\int \psi' \partial_{r}V |u|^{2} = -\Re \int f\nabla \psi \cdot \overline{\D u} -\frac{1}{2}\Re\int f \Delta\psi\bar{u},\notag
\end{align}
where $D^{2} \psi$ denotes the Hessian of $ \psi$, while $(\D)_{j} = \partial_{j} + iA_{j}$.
\end{lem}

A precise combination of these integral identities implies the following key identity that plays a fundamental role for proving the relevant estimates of this work.

\begin{pro}\label{keypr}
For $d\geq 1$, let $\psi: \Rd \longmapsto \mathbb{R}$ a regular radial function so that there exist $C>0$, $k\geq 0$ such that $|\psi^{(k)}| \leq C$ and $|\psi'(r)| \leq r$ if $r\leq 1$. Then, any solution $u \in H^1_{A}(\mathbb{R}^d) $ of the Helmholtz equation (\ref{magneticequation}) satisfies
\begin{align}\label{magneticidentity}
&\frac{1}{2}\int\psi'' |\D^{r}u \mp i\lambda^{1/2}u|^{2}+ \int \left(\frac{\psi'}{|x|} - \frac{\psi''}{2} \right)|\D^{\bot}u|^{2}\\
&+\Re\frac{(d-1)}{2}\int \nabla\left(\frac{\psi'}{|x|}\right)\cdot\D u\bar{u}+ \frac{\varepsilon}{2\lambda^{1/2}}\int \psi'\left|\D u \mp i\lambda^{1/2}\frac{x}{|x|}u\right|^{2}\notag\\
&- \Im\int \psi' \bar{u} B_{\tau}\cdot \D u + \frac{1}{2}\int (\psi'' V + \psi' \partial_{r}V)|u|^{2}  \notag\\
&+\frac{\varepsilon}{2\lambda^{1/2}}\Re\int \psi''\D^{r}u\bar{u}  - \frac{\varepsilon}{2\lambda^{1/2}}\int \psi'V|u|^{2}\notag\\
& = -\frac{\varepsilon}{2\lambda^{1/2}}\Re\int\psi' f\bar{u}-\Re \int f\psi' (\D^{r}\bar{u}\pm i\lambda^{1/2}\bar{u}) - \frac{(d-1)}{2}\Re\int \frac{\psi'}{|x|}f\bar{u}.\notag
\end{align}
\end{pro}

\begin{proof}

The proof consists in the combination of the above identities. We first compute
$$
(\ref{(4.3)}) + (\ref{(4.11)}) + \lambda^{1/2}(\ref{(4.2)})
$$
putting $\varphi=\frac{1}{2}\psi''$ in (\ref{(4.11)}). Then since $\psi$ is radial, letting $r=|x|$ yields
\begin{equation}\label{hessianoa}
\D u \cdot D^{2}\psi \cdot \overline{\D u} = \psi''|\D^{r}u|^{2} + \frac{\psi'}{r}|\D^{\bot}u|^{2}
\end{equation}
and
\begin{equation}\label{lapl}
\Delta\psi = \frac{(d-1)\psi'}{r} + \psi''.
\end{equation}
Thus we obtain
 \begin{align}
&\frac{1}{2}\int\psi''(|\D^{r}u|^{2} + \lambda|u|^{2}) \mp \Im\lambda^{1/2}\int \psi''\D^{r} u\bar{u}+ \int \left(\frac{\psi'}{r} - \frac{\psi''}{2} \right)|\D^{\bot}u|^{2}\notag\\
&+\Re\frac{d-1}{2}\int \nabla\left(\frac{\psi'}{r}\right)\cdot \D u\bar{u} \pm \varepsilon\lambda^{1/2}\int \psi'|u|^{2} \mp \varepsilon\Im \int \psi' \frac{x}{|x|}\cdot\D u\bar{u} \notag\\
& - \frac{d-1}{2}\int \frac{\psi'}{r} V|u|^{2} -\Re\int V\psi'\frac{x}{|x|}\cdot\D u\bar{u} - \Im\int \psi' \bar{u} B_{\tau}\cdot \D u\notag\\ 
&=  - \Re\int f \psi'\frac{x}{|x|}\cdot\D\bar{u}- \frac{(d-1)}{2}\Re\int \frac{\psi'}{|x|}f\bar{u} \pm \Im \lambda^{1/2}\int \psi' f\bar{u}.\notag
\end{align}

Note that above we have used that the integral
$$ \Im\int \psi' \bar{u} B_{\tau}\cdot \D u $$
is finite due to the condition \eqref{assB} and the property that $|\psi'(r)| \leq r$ if $r\leq 1$. The same reasoning applies to the terms involving $V$, which after some integration by parts this terms can be written as
$$\int \psi' \partial_{r}V|u|^{2} $$
and using in this case condition (\ref{assV1}) is finite.

Let us subtract now the identity $(\ref{(4.11)})$ multiplied by $\varepsilon$ from the above equality with the choice of the test function
\begin{equation}
\varphi=\frac{1}{2\lambda^{1/2}}\psi'.
\end{equation}
Hence, from the fact that
\begin{equation}\label{cuadrado}
\left|\D u \mp i \lambda^{1/2}\frac{x}{|x|}u \right|^{2} = |\D u|^{2} + \lambda|u|^{2} \mp 2\lambda^{1/2}\Im \frac{x}{|x|}\cdot \D u \bar{u},
\end{equation}
we get the square related to $\varepsilon$, and we conclude (\ref{magneticidentity}).
\end{proof}

\section{Proof of Theorem \ref{Theorem1}}\label{section2}

The proof will be divided into two parts depending on the relation between $\varepsilon$ and $\lambda$. For simplicity, we will work with the equation
\begin{equation}\label{magneticequation6}
(\nabla + iA)^{2}u + Vu + (\lambda + i\varepsilon) u = f, \quad \quad \quad \varepsilon >0.
\end{equation}
The same reasoning applies to the case $\lambda - i\varepsilon$.

\vspace{0.2cm}
We begin with the case when $0<\varepsilon < \lambda$.

\begin{pro}\label{alpha0magnetico} Let $d\geq 2$, $0<\varepsilon <\lambda$, $f$ such that $\Vert |\cdot |f\Vert_{L^{2}}<\infty$. Assume that {(H1)}, {(H2)} hold. In addition, let $A$ be continuous on $\R^{d}\backslash \{0\}$ and $curl \, A \in L^{1}_{loc}(\R^{d}\backslash \{0\})$ if $d=2$. Then, solutions $u\in H^{1}_{A}(\Rd)$ of the Helmholtz equation (\ref{magneticequation6}) satisfy
\begin{align}\label{magneticsommerfeld0}
\int \left|\D u - i\lambda^{1/2}\frac{x}{|x|}u\right|^{2} + \frac{\varepsilon}{\lambda^{1/2}} \int |x|\left|\D u -i\lambda^{1/2}\frac{x}{|x|}u \right|^{2} \leq C \int |x|^{2}|f|^{2} ,
\end{align}
where $C>0$ is independent of $\varepsilon, \lambda$.
\end{pro}

\begin{proof}
The proof is based on the equality (\ref{magneticidentity}) given in Proposition \ref{keypr}. Let us denote $r=|x|$. We define 
$$
\psi(r) = \frac{r^{2}}{2}
$$ 
so that $\psi'(r)=r$, $\psi''(r)= 1$ and we put it into the identity (\ref{magneticidentity}). Thus by (\ref{cuadrado}) it follows that
\begin{align}\label{propositionjuan}
&\frac{1}{2}\int \left|\D u - i\lambda^{1/2}\frac{x}{|x|}u \right|^{2} + \frac{\varepsilon}{2\lambda^{1/2}}\int |x|\left|\D u - i\lambda^{1/2}\frac{x}{|x|}u \right|^{2}\\
&= \Im \int |x|B_{\tau} \cdot  \D u \bar{u} -\frac{1}{2}\int (\partial_{r}(rV))|u|^{2}+ \frac{\varepsilon}{2\lambda^{1/2}}\int |x|V|u|^{2}\notag\\
& + \frac{\varepsilon(d-1)}{4\lambda^{1/2}} \int \frac{|u|^{2}}{|x|} -\frac{\varepsilon}{2\lambda^{1/2}}\Re\int |x|f\bar{u}-\frac{(d-1)}{2}\Re\int f\bar{u}\notag\\
& -\Re\int |x|f\left( \overline{\D^{r}u} + i\lambda^{1/2}\bar{u}\right).\notag
\end{align}

Let $v=e^{-i\lambda^{1/2}|x|}u$ and observe that $|v| = |u|$, $|\nabla_{A} v| = \left|\nabla_{A} u -i\lambda^{1/2}\frac{x}{|x|}u\right|$. Let us estimate the right-hand side of (\ref{propositionjuan}). We start with the observation that
\begin{equation}\label{btau}
B_{\tau}\cdot \D u = B_{\tau} \cdot \left( \D u -i \lambda^{1/2}\frac{x}{|x|}u\right).
\end{equation}
Hence by Cauchy-Schwarz inequality and (H2), we have
\begin{align}
\Im\int |x|B_{\tau}\cdot\D u \bar{u}& \leq \left(\int |x|^{2}|B_{\tau}|^{2}|v|^{2} \right)^{1/2}\left(\int |\D v|^{2} \right)^{1/2}\notag\\
& < A_{B}\int |\D v|^{2}.\notag
\end{align}
Similarly, we get
\begin{align}
-\frac{1}{2}\int (\partial_{r}(rV))|u|^{2} & \leq \frac{1}{2}\int (\partial_{r}(rV))_{-}|u|^{2}\notag\\
&< \frac{A_{V}}{2}\int |\D v|^{2}.\notag
\end{align}
Let us observe now that  from assumption (H1) it follows that 
\begin{align}\label{H3}
\int |x|V|u|^{2} < \int |x||\nabla_{A} u|^{2}.
\end{align}
This combining with (H2), yields
\begin{align}
\frac{\varepsilon}{2 \lambda^{1/2}}\int |x|V|u|^{2} <\frac{\varepsilon A_{V}}{2\lambda^{1/2}}\int |x||\D v|^{2}\notag.
\end{align}
Let us now compute the term $\frac{\varepsilon}{\lambda^{1/2}}\int \frac{|u|^{2}}{|x|}$. To this end, let $\delta > 0$. Then since $\varepsilon < \lambda$ and by the a-priori estimate (\ref{a+++0}), we have
\begin{align}\label{landaepsilon}
\frac{\varepsilon}{\lambda^{1/2}}\int \frac{|u|^{2}}{|x|} &= \frac{\varepsilon}{\lambda^{1/2}}\int_{|x|<\frac{\lambda^{1/2}\delta}{\varepsilon}} \frac{|u|^{2}}{|x|} + \frac{\varepsilon}{\lambda^{1/2}}\int_{|x|\geq \frac{\lambda^{1/2}\delta}{\varepsilon}}\frac{|u|^{2}}{|x|}\\
& \leq \delta\int \frac{|u|^{2}}{|x|^{2}} + \frac{\varepsilon}{\delta}\int |u|^{2}\notag\\
& \leq \delta\int \frac{|u|^{2}}{|x|^{2}} + \frac{1}{\delta}\int |f||u|. \notag
\end{align}
Consequently, by Cauchy-Schwarz inequality and the magnetic Hardy inequality (\ref{A.111}) if $d\geq 3$ and (\ref{Hardydimension2}) when $d=2$, it follows that
\begin{align}
\frac{\varepsilon}{\lambda^{1/2}}\int \frac{|u|^{2}}{|x|} & \leq (\delta + \kappa)\int \frac{|u|^{2}}{|x|^{2}} + \frac{1}{4\kappa\delta^{2}}\int |x|^{2}|f|^{2}\notag\\
& \leq C_{H}(\delta + \kappa) \int |\D v|^{2} + C_{\kappa, \delta}\int |x|^{2}|f|^{2},\notag
\end{align}
for $\kappa >0$, where $C_{H}= \frac{4}{(d-2)^{2}}$ if $d\geq 3$ and $C_{H} = C_{H_{2}}$ when $d=2$. Similar arguments apply to estimate the terms containing $f$. By $\varepsilon < \lambda$, the a-priori estimate (\ref{a+++0}) and
\begin{equation} \label{juan1}
\D^{r}\bar{u} + i\lambda^{1/2}\bar{u} = \overline{\left(\D u - i\lambda^{1/2}\frac{x}{|x|}u \right)}\cdot \frac{x}{|x|},
\end{equation}
we deduce
\begin{align}
-\frac{\varepsilon}{2\lambda^{1/2}}\Re\int |x| f\bar{u} &\leq \frac{\varepsilon^{3/2}}{4\lambda^{1/2}}\int |u|^{2}+ \frac{\varepsilon^{1/2}}{4\lambda^{1/2}}\int |x|^{2}|f|^{2}\notag\\
& \leq \frac{\varepsilon}{4}\int |u|^{2} +\frac{1}{4}\int |x|^{2}|f|^{2}\notag\\
& \leq \kappa\int \frac{|u|^{2}}{|x|^{2}} + C_{\kappa}\int |x|^{2}|f|^{2}\notag\\
& \leq  C_{H} \kappa \int |\D v|^{2} + C_{\kappa}\int |x|^{2}|f|^{2},\notag
\end{align}
\begin{align}
-\frac{(d-1)\Re}{2}\int f\bar{u} \leq C_{H}\kappa \int |\D v|^{2} + C_{\kappa}\int |x|^{2}|f|^{2},\notag
\end{align}
\begin{align}
-\Re \int |x|f \left(\overline{\D^{r}u} +i\lambda^{1/2}\bar{u} \right)& \leq \kappa\int |\D v|^{2} + C_{\kappa}\int |x|^{2}|f|^{2},\notag
\end{align}
for arbitrary $\kappa>0$, being $C_{H}$ as above.

Thus it may be concluded that
\begin{align}
\frac{1}{2}&\int |\D v|^{2} +\frac{\varepsilon}{2\lambda^{1/2}}\int|x||\D v|^{2} \notag\\
&< \left(\frac{2A_{B}+A_{V}}{2} +C_{H} (\delta + 3\kappa)  + \kappa \right)\int |\D v|^{2} + \frac{\varepsilon A_{V}}{2\lambda^{1/2}}\int |x||\D v|^{2}\notag\\
& + C \int |x|^{2}|f|^{2}.\notag
\end{align}
Note that since $u\in H^{1}_{A}(\Rd)$, by Remark \ref{2+} it is a simple matter to check that the right-hand side of the above inequality is finite. Therefore, choosing $\kappa, \delta$ small enough and using that $A_{V} + 2A_{B}<1$, (\ref{magneticsommerfeld0}) is proved.
\end{proof}

\begin{remark}
Note that the identity (\ref{propositionjuan}) is true for any $d\geq 1$. However, since the magnetic  Hardy inequalities (\ref{A.111}) and (\ref{Hardydimension2}) are valid only in the three or higher dimensional case and in dimension two, respectively, the method of proof breaks down when $d=1$. In the one dimensional case, the main difficulty is about the analysis of the term $\frac{\varepsilon}{\lambda^{1/2}}\Re\int |x|f\bar{u}$. 
\end{remark}

The proof will be completed by showing the following estimate.

\begin{pro} Let $d\geq 2$, $\lambda \leq \varepsilon$, $f$ such that $\Vert |\cdot| f \Vert_{L^{2}} < \infty$. Assume that (H1) holds. Moreover, let $A$ be continuous on $\R^{d}\backslash \{0\}$ and $curl \, A \in L^{1}_{loc}(\Rd \backslash \{0\})$ if $d=2$. Then the solution $u\in H^{1}_{A}(\Rd)$ of the equation (\ref{magneticequation6}) satisfies
\begin{equation}\label{bai}
\int |\D u|^{2} \leq C\int |x|^{2}|f|^{2}.
\end{equation}
\end{pro}

\begin{proof}
For this purpose, let us multiply equation (\ref{magneticequation6}) by $\bar{u}$ and integrate over $\Rd$. Then, taking the real part yields
\begin{equation}\notag
\int |\D u|^{2} = \lambda\int |u|^{2} + \int V|u|^{2} - \Re \int f\bar{u}.
\end{equation} 
Since $\lambda \leq \varepsilon$, by assumption (H1) and the a-priori estimate (\ref{a+++0}), we have
\begin{align}
\int |\D u|^{2} & \leq 2\int |f||u|\notag\\
& \leq \left(\int \frac{|u|^{2}}{|x|^{2}} \right)^{1/2}\left(\int |x|^{2}|f|^{2} \right)^{1/2},\notag
\end{align}
and the result follows using the magnetic Hardy inequality \eqref{A.111} when $d\geq 3$ and (\ref{Hardydimension2}) if $d=2$.
\end{proof}

\section{Proof of Theorem \ref{TheoremSRC}}\label{section3}

We now proceed to show the sharp Sommerfeld radiation condition (\ref{keySRC}). In order to get this inequality, it will be necessary to control the Agmon-H\"ormander norm of solutions $u$ of the electromagnetic Helmholtz equation (\ref{magneticequation6}), as well as the same norm for the magnetic gradient $\D u$ and the term $\int \frac{|\D^{\bot}u|^{2}}{|x|}$ related to the tangential component of the magnetic gradient. We will work with potentials such that for any $\lambda \geq \lambda_{0} >0$ solutions $u\in H^{1}_{A}(\Rd)$ of the equation (\ref{magneticequation6}) satisfy the following a-priori estimate
\begin{equation}\label{Morrey}
\lambda|||u|||_{1}^{2} + |||\D u|||_{1}^{2} + \int \frac{|\D^{\bot} u|^{2}}{|x|} \leq C(1+\varepsilon) (N_{1}(f))^{2},
\end{equation}
where $C=C(\lambda_{0}) >0$.

\begin{remark}
Under the hypotheses of Theorem \ref{TheoremSRC}, estimate (\ref{Morrey}) can be proved following \cite{Z}, \cite{F} or \cite{PV1}. In addition, using the same arguments as in \cite{IS} and in \cite{Z}, from Theorem \ref{Theorem1} follows
\begin{equation}\notag
\lambda |||u|||_{1} + |||\D u|||_{1}^{2} \leq C(\lambda_{0}) \int |x|^{2}|f|^{2}.
\end{equation}
\end{remark}

As a warm up we  give the $\alpha =1$ version of (\ref{somalpha})  for the constant coefficient case. 

\subsection{$\alpha=1$ version for the constant coefficient case ($A\equiv 0 \equiv V$)}

Let us consider the Helmholtz equation
\begin{equation}\label{1}
\Delta u + \lambda u + i\varepsilon u = f.
\end{equation}
Then we can state the following inequality.

\begin{lem}\label{lemmaalpha1}
Let $d\geq 1$, $\lambda >0$, $\varepsilon >0$ and $f$ such that $\Vert |\cdot|^{3/2}f\Vert_{L^{2}}<\infty$, $\Vert |\cdot|^{2}f\Vert_{L^{2}}<\infty$. Then any solution $u\in H^{1}(\Rd)$ of the equation (\ref{1}) satisfies
\begin{align}\label{sommerfeld1}
\frac{1}{4}\int |x| &\left|\partial_{r} u - i\lambda^{1/2}u + \frac{d-1}{2|x|}u\right|^{2} +
 \frac{\varepsilon}{4\lambda^{1/2}}\int |x|^{2}\left|\nabla u - i\lambda^{1/2}\frac{x}{|x|}u\right|^{2}\\
& \leq \frac{1}{4} \int
|x|^{3} |f|^{2} + \frac{d}{4\lambda^{1/2}} \int
|f||u| + \frac{\varepsilon}{4\lambda^{1/2}}\int |x|^{2}|f||u|,\notag
\end{align}
\end{lem}

\begin{proof}
The proof is based on the analogue identity of (\ref{magneticidentity}) for $A\equiv 0$, $V=0$. Note that in this case, $\D \equiv \nabla$. Thus we have
\begin{align}\label{constantidentity}
&\frac{1}{2}\int\psi''(|\partial_{r}u|^{2} + \lambda|u|^{2}) - \Im\lambda^{1/2}\int \psi''\partial_{r}u\bar{u}+ \int \left(\frac{\psi'}{|x|} - \frac{\psi''}{2} \right)|\nabla^{\bot}u|^{2}\\
&+\Re\frac{(d-1)}{2}\int \nabla\left(\frac{\psi'}{|x|}\right)\cdot\nabla u\bar{u}+ \frac{\varepsilon}{2\lambda^{1/2}}\int \psi'\left|\nabla(e^{-i\lambda^{1/2}|x|}u)\right|^{2} \notag\\
&+\frac{\varepsilon}{2\lambda^{1/2}}\Re\int \psi''\partial_{r}u\bar{u} = -\frac{\varepsilon}{2\lambda^{1/2}}\Re\int\psi' f\bar{u}-\Re \int f\psi' (\partial_{r}\bar{u}+i\lambda^{1/2}\bar{u})\notag\\
& - \frac{(d-1)}{2}\Re\int \frac{\psi'}{|x|}f\bar{u}.\notag
\end{align}
Let us define
$$
\psi(r)=\frac{r^{3}}{3}, \quad \quad r=|x|
$$
and we put it into (\ref{constantidentity}). Hence, since
\begin{align}
\left|\partial_{r}u - i\lambda^{1/2}u +\frac{(d-1)}{2|x|}u\right|^{2} &= |\partial_{r}u|^{2} + \lambda|u|^{2} + \frac{(d-1)^{2}}{4|x|^{2}}|u|^{2}\notag \\
& - 2\lambda^{1/2}\Im \partial_{r}u \bar{u} + \Re \frac{(d-1)}{2|x|}\partial_{r}u\bar{u},\notag
\end{align}
we obtain
\begin{align}\label{identityalpha1}
&\frac{1}{2}\int_{\Rd} |x|\left|\partial_{r}u - i\lambda^{1/2}u +\frac{(d-1)}{2|x|}u\right|^{2} + \frac{\varepsilon}{4\lambda^{1/2}}\int_{\Rd} |x|^{2}\left|\nabla u - i\lambda^{1/2}\frac{x}{|x|}u\right|^{2}\\
& =\frac{d\varepsilon}{4\lambda^{1/2}}\int_{\Rd}|u|^{2}  -\frac{1}{2}\Re \int f|x|^{2}\left(\partial_{r}\bar{u}+i\lambda^{1/2}\bar{u} +\frac{(d-1)}{2}\bar{u}\right)-\frac{\varepsilon}{4\lambda^{1/2}}\Re\int_{\Rd}|x|^{2} f\bar{u},\notag
\end{align}
and by (\ref{a+++0}) the lemma follows.
\end{proof}

\begin{remark}
It is not our purpose to give the corresponding estimate for the electromagnetic case. However, under suitable assumptions on the potentials, the analogue of (\ref{sommerfeld1}) may be obtained in much the same way as in the constant coefficient case.
\end{remark}

\subsection{Sharp Sommerfeld condition}

We are now in a position to show the goal of this section.

\begin{pro}\label{som}
Let $d\geq 3$, $\lambda_{0}, \varepsilon>0$ and $f$ such that $N(f)<\infty$ and $\Vert |\cdot|^{2}f\Vert_{L^{2}}<\infty$. Let the potentials satisfy {\bf(H3)}. Then, there exists positive constant $C= C(\lambda_{0})$ such that for any $R\geq 1$ and $\lambda \geq \lambda_{0}$ the solution $u\in H^{1}_{A}(\Rd)$ to the Helmholtz equation (\ref{magneticequation6}) satisfies
\begin{align}\label{keymagnetic}
&\int_{|x|\leq R} |x|\left|\D^{r}u -i \lambda^{1/2}u + \frac{(d-1)}{2|x|}u \right|^{2} + R \int_{|x|\geq 2R} |\nabla_{A}(e^{-i\lambda^{1/2}|x|}u)|^{2}\\ 
&+\frac{\varepsilon}{\lambda^{1/2}}\int_{|x|\leq R}|x|^{2}|\nabla_{A}(e^{-i\lambda^{1/2}|x|}u)|^{2} + \frac{\varepsilon R}{\lambda^{1/2}}\int_{|x|\geq 2R}|x||\nabla_{A}(e^{-i\lambda^{1/2}|x|}u)|^{2}\notag\\
& \leq C\left[ |||u|||_{1}^{2} + |||\D u|||_{1} + \int \frac{|\D^{\bot}u|^{2}}{|x|} + \frac{(N_{1}(f))^{2}}{\lambda} +  \int (1+\varepsilon |x|)|x|^{3} |f|^{2} \right]\notag.
\end{align}
\end{pro}

\begin{proof}
Let $R \geq 1$ and observe that 
\begin{align}
|||u|||_{R}^{2} & \geq \frac{1}{2R}\int_{|x|\leq 2R} |u|^{2} \geq \frac{1}{2R} \int_{R\leq |x| \leq 2R} |u|^{2} = \frac{1}{2R}\int_{R}^{2R} \frac{1}{r}\int_{|x|=r} |x||u|^{2}\notag\\
&  \geq \frac{\log 2}{2} \inf_{R\leq r \leq 2R} \int_{|x|=r} |u|^{2},\notag
\end{align}
which implies that there exists $R_{1}$ such that $R \leq R_{1} \leq 2R$ and satisfies
\begin{align}
& \int_{|x|=R_{1}} |u|^{2} \leq C|||u|||_{R_{1}}^{2},\label{boundary1}\\
& \int_{|x|=R_{1}} |x||u|^{2} \leq C\int |u|^{2},\label{boundary2}
\end{align}
with $C>1$. We will prove estimate (\ref{keymagnetic}) for this $R_{1}$, and then, since $R_{1}$ and $R$ are comparable, we will deduce the result for any $R\geq 1$.

Let us define the multiplier
\begin{equation}\notag
\psi'(|x|) = \left\{ \begin{array}{ll}
|x|^{2} & \textrm{if $|x| \leq 1$},\vspace{0.1cm} \\
|x|& \textrm{if $|x| \geq 1$},
\end{array} \right.
\end{equation}
and set $\psi'_{R_{1}}(|x|) = R_{1}^{2}\psi'\left(\frac{|x|}{R_{1}}\right)$. Thus we get
\begin{equation}\notag
\psi'_{R_{1}}(|x|) = \left\{ \begin{array}{ll}
|x|^{2} & \textrm{if $|x| \leq R_{1}$},\vspace{0.1cm} \\
R_{1}|x|& \textrm{if $|x| \geq R_{1}$},
\end{array} \right.
\end{equation}
so that in the distributional sense yields
\begin{equation}\notag
\psi''_{R_{1}}(|x|) = \left\{ \begin{array}{ll}
2|x| & \textrm{if $|x| \leq R_{1}$},\vspace{0.1cm} \\
R_{1}& \textrm{if $|x| \geq R_{1}$}.
\end{array} \right.
\end{equation}

Let us put $\psi'_{R_{1}}$ into the identity (\ref{magneticidentity}). For simplicity, we start by considering the case when $A_{j}=V=0$, $j=1,\ldots,d$. Thus denoting $v=e^{-i\lambda^{1/2}|x|}u$ we get 
\begin{align}\label{keyidentity}
&\int_{|x|\leq R_{1}} |x|\left|\D^{r} u -i\lambda^{1/2}u + \frac{(d-1)}{2|x|}u\right|^{2} + \frac{R_{1}}{2}\int_{|x|\geq R_{1}} \left|\D u -i\lambda^{1/2}\frac{x}{|x|}u\right|^{2}\\
& + \frac{\varepsilon}{2\lambda^{1/2}}\left[ \int_{|x|\leq R_{1}} |x|^{2}\left|\D v\right|^{2} + R_{1}\int_{|x|\geq R_{1}} |x|\left|\D v\right|^{2}\right]\notag\\
& = \frac{d\varepsilon}{2\lambda^{1/2}}\int_{|x|\leq R_{1}} |u|^{2} + \frac{\varepsilon(d-1)R_{1}}{4\lambda^{1/2}}\int_{|x|\geq R_{1}} \frac{|u|^{2}}{|x|} \notag\\
&   -\frac{\varepsilon R_{1}}{2\lambda^{1/2}}\Re\int_{|x|\geq R_{1}}|x|f\bar{u} - \Re\int_{|x|\leq R_{1}} f|x|^{2}\left(\D^{r}\bar{u} +i\lambda^{1/2}\bar{u} +\frac{d-1}{2|x|}\bar{u} \right)\notag\\
& -\Re R_{1}\int_{|x|\geq R_{1}} |x|f(\D^{r}\bar{u} + i\lambda^{1/2}\bar{u}) -\frac{(d-1)R_{1}}{2}\Re\int_{|x|\geq R_{1}} f\bar{u}\notag\\
& - \frac{\varepsilon}{2\lambda^{1/2}}\Re\int_{|x|\leq R_{1}}|x|^{2}f\bar{u}+\frac{(d-1)}{4}\int_{|x|=R_{1}} |u|^{2} + \frac{\varepsilon}{4\lambda^{1/2}}\int_{|x|=R_{1}}|x||u|^{2}.\notag
\end{align}

Let us analyze now the right hand side of the above equality. By the a-priori estimate (\ref{a+++0}), the $\varepsilon$ terms can be upper bounded by the following observation
\begin{align}\notag
\frac{\varepsilon}{\lambda^{1/2}}\int |u|^{2} + \frac{\varepsilon}{\lambda^{1/2}} \int |x|^{2}|f||u| & \leq C \left(|||u|||_{1}^{2} + \frac{(N_{1}(f))^{2}}{\lambda} + \frac{\varepsilon}{\lambda^{1/2}}\int |x|^{4}|f|^{2} \right).
\end{align}
In addition, it is easy to check that
\begin{align}
R_{1}\int_{|x|\geq R_{1}} |f||u| & \leq R_{1} \sum_{2^{j} \geq R_{1}} \int_{C(j)} |f||u|\notag\\
& \leq C\left(\int |x|^{3}|f|^{2} \right)^{1/2}|||u|||_{1}.\notag
\end{align}
By using (\ref{juan1}), for $\kappa >0$ we get
\begin{align}\notag
\left| R_{1}\int_{|x|\geq R_{1}} |x| f (\D^{r}\bar{u} + i\lambda^{1/2}\bar{u}) \right| \leq \kappa R_{1}\int_{|x|\geq R_{1}} |\D v|^{2} + C(\kappa)\int |x|^{3}|f|^{2},
\end{align}
\begin{align}
\left|\int_{|x|\leq R_{1}} f|x|^{2}\left(\D^{r}\bar{u} + i\lambda^{1/2}\bar{u} + \frac{d-1}{2|x|}\bar{u} \right) \right| &\leq \kappa \int_{|x|\leq R_{1}} |x|\left| \D^{r}u -i\lambda^{1/2} u + \frac{d-1}{2|x|}u \right|^{2}\notag\\
& + C(\kappa)\int |x|^{3}|f|^{2}. \notag
\end{align}
The surface integrals pose no problem. By (\ref{boundary1}) and (\ref{boundary2}) we have
\begin{align}
\int_{|x|=R_{1}} |u|^{2}& + \frac{\varepsilon}{\lambda^{1/2}} \int_{|x|=R_{1}}|x||u|^{2} \leq C|||u|||_{R_{1}}^{2} + \frac{C\varepsilon}{\lambda^{1/2}} \int |u|^{2}\notag\\
& \leq C|||u|||_{R_{1}}^{2} + C|||u|||_{1}\frac{N_{1}(f)}{\lambda^{1/2}}. \notag
\end{align}
Thus, from the above estimates and choosing $\kappa >0$ small enough, we get (\ref{keymagnetic}) for the free case.

We now turn to the case when $A \neq 0$ and $V\neq 0$. After substituting the above multiplier into the identity (\ref{magneticidentity}), in the right-hand side of the resulting equality the integrals related to $B_{\tau}$ and $V$ are
\begin{align}
&\Im\int_{|x|\leq R_{1}} |x|^{2}B_{\tau}\cdot \overline{\D u}u - \Im R_{1}\int_{|x|\geq R_{1}} |x|B_{\tau}\cdot \D u\bar{u} - \int_{|x|\leq R_{1}} |x|V|u|^{2} \notag\\
& -\frac{1}{2}\int_{|x|\leq R_{1}} |x|^{2}(\partial_{r}V)|u|^{2} -\frac{R_{1}}{2}\int_{|x|\geq R_{1}} V|u|^{2} - \frac{R_{1}}{2}\int_{|x|\geq R_{1}} |x|(\partial_{r}V)|u|^{2}\notag\\
& +\frac{\varepsilon}{2\lambda^{1/2}}\int_{|x|\leq R_{1}} |x|^{2}V|u|^{2} + \frac{\varepsilon R_{1}}{2\lambda^{1/2}}\int_{|x|\geq R_{1}} |x|V|u|^{2}. \notag
\end{align}
Let us treat the above terms. We start with the magnetic ones. Noting that
$$
B_{\tau} \cdot \D u = B_{\tau} \cdot \D^{\bot} u,
$$
by Cauchy-Schwarz inequality and (H3) we get
\begin{align}
\Im\int_{|x|\leq R_{1}}& |x|^{2}B_{\tau}\cdot \overline{\D u}u + \Im R_{1}\int_{|x|\geq R_{1}}|x|B_{\tau}\cdot\overline{\D u}u\notag\\
&  \leq \left( \int_{|x|\leq 1} |\D u|^{2}\right)^{\frac{1}{2}}\left( \int_{|x|\leq 1} |x|^{4}|B_{\tau}|^{2}|u|^{2} \right)^{\frac{1}{2}} + \int_{|x|\geq 1} |x|^{2}|B_{\tau}||\D^{\bot} u| |u| \notag\\
& \leq C \left( ||| \D u|||_{1}^{2} +\int \frac{|\D^{\bot} u|^{2}}{|x|} + |||u|||_{1}^{2} \right).\notag
\end{align}
As far as the electric potential is concerned, note that after integrating by parts the terms containing $\partial_{r}V$ and using that $\Re \nabla \bar{u} = \Re \D u \bar{u}$, one can rewrite them as follows
\begin{align}
&\frac{d-1}{2}\int_{|x|\leq R_{1}} |x|V|u|^{2} + \Re\int_{|x|\leq R_{1}} V|x|^{2}\frac{x}{|x|}\cdot \D u \bar{u}\notag\\
&+\frac{(d-1)R_{1}}{2}\int_{|x|\geq R_{1}}V|u|^{2} + \Re R_{1}\int_{|x|\geq R_{1}} V|x|\frac{x}{|x|}\cdot\D u\bar{u}\notag\\
&+\frac{\varepsilon}{2\lambda^{1/2}}\int_{|x|\leq R_{1}}|x|^{2}V|u|^{2} + \frac{\varepsilon R_{1}}{2\lambda^{1/2}}\int_{|x|\geq R_{1}}|x|V|u|^{2}\notag.
\end{align}
Then by Cauchy-Schwarz inequality, the following magnetic Hardy type inequality
\begin{equation}\notag
\int_{|x|\leq R} \frac{|f|^{2}}{|x|^{2}}\leq \frac{4}{(d-2)^{2}}\int_{|x|\leq R} |\D f|^{2} + \frac{2}{(d-2)R}\int_{|x|=R} |f|^{2}d\sigma_{R} \quad \quad R>0,
\end{equation}
using that $\int_{|x|=1} |u|^{2} \leq 4 |||u|||_{1}^{2}$ and condition (H3),  we have
\begin{align}
\int |x||V||u|^{2} & \leq \int_{|x|\leq 1} \frac{|u|^{2}}{|x|^{2}} + \int_{|x|\geq 1} \frac{|u|^{2}}{|x|^{2+\alpha}}\notag\\
& \leq C\left( \int_{|x|\leq 1} |\D u|^{2} + \int_{|x|=1} |u|^{2} + \sum_{j\geq 0} 2^{-j(1+\alpha)} |||u|||_{1}^{2}\right)\notag\\
& \leq C \left(|||\D u|||_{1}^{2} + |||u|||_{1}^{2}\right),\notag
\end{align}
for some $C>1$. Similarly, yields
\begin{align}
\int |V||x|^{2}|\D u| |u| & \leq C \left( \int_{|x|\leq 1} |\D u||u| + \int_{|x|\geq 1} \frac{|\D u| |u|}{|x|^{1+\alpha}}\right)\notag\\
& \leq C (|||\D u |||_{1}^{2} + |||u|||_{1}^{2})\notag 
\end{align}
and by the a-priori estimate (\ref{a+++0}), we obtain
\begin{align}
\frac{\varepsilon}{\lambda^{1/2}}\int |x|^{2}|V||u|^{2} & \leq \frac{C}{\lambda^{1/2}}\int |f||u|\notag\\
& \leq \frac{C}{\lambda^{1/2}}N_{1}(f)|||u|||_{1}.\notag
\end{align}
As a consequence, it follows that the terms containing the electric potential $V$ are upper bounded by
\begin{align}
&\frac{d-1}{2}\int |x||V||u|^{2} + \int |V||x|^{2}|\D u||u| + \frac{\varepsilon}{2\lambda^{1/2}}\int |x|^{2}|V||u|^{2}\notag\\
&\leq C\left(|||u|||_{1}^{2} + |||\D u|||_{1}^{2} + \frac{1}{\lambda}(N_{1}(f))^{2} \right).\notag
\end{align}
Putting everything together the proposition follows.
\end{proof}

Combining this result with estimate (\ref{Morrey}), provides the following inequality which in particular proves Theorem \ref{TheoremSRC}.

\begin{cor}\label{somcorolario}
Under the hypotheses of Theorem \ref{TheoremSRC}, for any $R\geq 1$ the solution $u\in H^{1}_{A}(\Rd)$ of the Helmholtz equation (\ref{magneticequation6}) satisfies
\begin{align}\label{keymagnetic1}
&\int_{|x|\leq \frac{R}{2}} |x|\left|\D^{r}u -i \lambda^{1/2}u + \frac{(d-1)}{2|x|}u \right|^{2} + R \int_{|x|\geq R} |\nabla_{A}(e^{-i\lambda^{1/2}|x|}u)|^{2}\\ 
&+\frac{\varepsilon}{\lambda^{1/2}}\int_{|x|\leq \frac{R}{2}}|x|^{2}|\nabla_{A}(e^{-i\lambda^{1/2}|x|}u)|^{2} + \frac{\varepsilon R}{\lambda^{1/2}}\int_{|x|\geq R}|x||\nabla_{A}(e^{-i\lambda^{1/2}|x|}u)|^{2}\notag\\
& \leq C\left[\int |x|^{3}|f|^{2} +(1+\varepsilon)(N_{1}(f))^{2} +\varepsilon \int |x|^{4} |f|^{2}\right],\notag
\end{align}
where $C=C(\lambda_{0})$ is independent of $\varepsilon$.
\end{cor}

\subsubsection{Proof of Theorem \ref{TheoremSRC}}
Following the same ideas as in \cite{Z} (see also \cite{IS}), if we denote $R(\lambda + i\varepsilon)f$ the solution of (\ref{magneticequation6}) satisfying (\ref{keymagnetic1}), it may be concluded that there exists
$$
u(x)= R(\lambda + i0)f = \lim_{\varepsilon \to 0^{+}} R(\lambda + i\varepsilon)f \quad \quad \text{in} \quad (H^{1}_{A})_{loc}
$$
such that $u$ is a solution of (\ref{resjuan}) that satisfies
\begin{align}
&\int_{|x|\leq \frac{R}{2}} |x|\left|\D^{r}u -i \lambda^{1/2}u + \frac{(d-1)}{2|x|}u \right|^{2} + R \int_{|x|\geq R} |\nabla_{A}(e^{-i\lambda^{1/2}|x|}u)|^{2}\\
& \leq C\left[\int |x|^{3}|f|^{2} +(N_{1}(f))^{2} \right],\notag
\end{align}
where $C=C(\lambda_{0})$ is independent of $R$. Now, taking the supremum over $R\geq 1$ we get (\ref{keySRC}) and the proof is complete. 

\begin{remark}\label{remarkboundary}
This result provides an extra a-priori estimate for the surface integral. In fact, the solution $u$ of the equation (\ref{magneticequation6}) holds
\begin{equation}\label{surfaceint}
\sup_{R\geq 1}\int_{|x|=R} |u|^{2}< \infty.
\end{equation}
Note that we can rewrite (\ref{keyidentity}) for any $R\geq 1$ with the boundary terms in the left hand side of the identity. Hence from (\ref{keymagnetic}) it is immediate that
\begin{align}
\int_{|x|=R} |u|^{2} &+ \frac{\varepsilon}{\lambda^{1/2}} \int_{|x|=R} |x||u|^{2} \leq C\left(\int (1+ \varepsilon|x| )|x|^{3}|f|^{2} + (1+\varepsilon)(N_{1}(f))^{2}\right)\notag,
\end{align}
where $C=C(\lambda_{0}) >0$ is independent of $\varepsilon$.
\end{remark}

\section{The forward problem}\label{section5}

This section is devoted to the study of the forward problem for the electromagnetic Helmholtz equation with singular potentials. Firstly, under the hypotheses of Theorem \ref{Theorem1}, we will prove the limiting absorption principle for the equation (\ref{resjuan}). Secondly, we proceed with the study of the cross-section of the solution of the electromagnetic Helmholtz equation (\ref{resjuan}). This will follow by the radiation condition (\ref{keySRC}) and the resolvent estimates (\ref{BP}), (\ref{Morrey}). Finally, the limiting absorption principle and the existence and uniqueness of the cross-section will allow us to give the spectral representation of $H_{A}$.

\subsection{Limiting absorption principle}\label{LAPchap4}

Let us first state a uniqueness theorem for the magnetic Schr\"odinger operator with potentials satisfying (H2) together with some extra assumption. 

Let us consider the homogeneous electromagnetic Helmholtz equation
\begin{equation}\label{homo}
(\nabla + iA)^{2}u + Vu + \lambda u = 0.
\end{equation}

\begin{thm}\label{unicidad3}
Let $d\geq 1$, $\lambda_{0} > 0$ and assume that {\bf(H2)} holds. Let $u$ be a solution of (\ref{homo}) with $u, \D u \in L^{2}_{loc}$ such that for any $\lambda \geq \lambda_{0}$ satisfies
\begin{equation}\label{Vuni}
\lim\inf_{R \to \infty}\int_{|x|=R} V|u|^{2}  = 0
\end{equation}
and
\begin{equation}\label{limit11}
\lim_{R\to \infty} \frac{1}{R}\int_{R\leq |x|\leq2 R} (\lambda |u|^{2} + |\D u|^{2}) = 0.
\end{equation}
Then $u\equiv 0$.
\end{thm}

\begin{proof}
 By (\ref{limit11}), there exists a sequence $\{R_{j}\}$ tending to infinity such that
\begin{equation}\label{limit12}
\lim_{R_{j}\to \infty} \frac{1}{R_{j}} \int_{R_{j} \leq |x| \leq 2R_{j}} (\lambda |u|^{2} + |\D u|^{2}) = 0.
\end{equation}
Let us multiply the equation (\ref{homo}) by the combination of the symmetric and the antisymmetric multipliers 
$$
\nabla\psi \cdot \overline{\D u} + \frac{1}{2}\Delta\psi \bar{u} + \varphi\bar{u},
$$
where $\psi, \varphi$ are a real valued functions and integrate over the ball $\{|x|<R_{j}\}$. Hence we have
\begin{align}\label{(4.31)}
&\int_{|x| < R_{j}} \D u\cdot D^{2}\psi \cdot \overline{\D u} -\int_{|x|<R_{j}}\varphi|\D u|^{2} + \int_{|x|<R_{j}} \varphi \lambda |u|^{2}\\
& -\frac{1}{4}\int_{|x| < R_{j}}( \Delta^{2}\psi -2\Delta\varphi) |u|^{2} + \int_{|x|<R_{j}}\varphi V|u|^{2} + \frac{1}{2}\int_{|x|<R_{j}}\nabla V\cdot \nabla\psi|u|^{2}\notag\\
& =\Im \sum_{k,m =1}^{d} \int_{|x| < R_{j}} \frac{\partial\psi}{\partial x_{k}} B_{km}u\overline{(\D)_{m}u}+ \frac{1}{4} \int_{S_{R_{j}}} \nabla(\Delta\psi)\cdot\frac{x}{|x|} |u|^{2} \notag\\
&  + \frac{1}{2}\Re \int_{S_{R_{j}}} \frac{x}{|x|}\cdot \overline{\D u} \Delta\psi u +\frac{1}{2}\int_{S_{R_{j}}}\frac{x}{|x|}\cdot \nabla\varphi |u|^{2}-\Re\int_{S_{R_{j}}}\D^{r} u\varphi\bar{u}\notag\\
& + \frac{1}{2} \int_{S_{R_{j}}} (\lambda +V)\frac{x}{|x|}\cdot\nabla\psi|u|^{2}  - \frac{1}{2}\int_{S_{R_{j}}} \frac{x}{|x|} \cdot\nabla\psi |\D u|^{2},\notag
\end{align}
being $S_{R_{j}} = \{|x|=R_{j}\}$.

Let $R$ be such that $1\leq \frac{R_{j}}{2} \leq R \leq R_{j}$ and we consider the following multipliers 
\begin{equation}\notag
\varphi(x)=\frac{1}{2R}, \quad \quad \quad \quad \quad
\nabla\psi(x)= \frac{x}{R}.
\end{equation} 
Let us insert them into the identity (\ref{(4.31)}). Noting that the boundary terms can be upper bounded by 
\begin{equation}\notag
C\int_{|x|=R_{j}} \{|\D u|^{2} + (\lambda + V)|u|^{2}\} d\sigma_{R_{j}},
\end{equation}
where $C=C(\lambda_{0})$, it follows that
\begin{align}
\frac{1}{2R}\int_{|x|\leq R_{j}} (\lambda|u|^{2} + |\D u|^{2}) & \leq \frac{1}{2R}\int_{|x|\leq R_{j}} (\partial_{r}(rV))_{-}|u|^{2}\label{pepe}\\
& + \frac{1}{R}\int_{|x|\leq R_{j}} |x||B_{\tau}||\D u||u|\notag\\
& + C\int_{|x|=R_{j}}(|\D u|^{2} + (\lambda + V) |u|^{2}).\notag
\end{align}

Let $\theta(r)\in C^{\infty}(\R)$ be a cut-off function such that $0\leq \theta \leq 1$, given by
\begin{equation}\notag
\theta(r) = \left\{ \begin{array}{ll}
1 & \textrm{if $r \leq 1$}\\
0 & \textrm{if $r \geq 2$}
\end{array} \right.
\end{equation}
and set $\theta_{R} = \theta\left(\frac{|x|}{R} \right)$. Hence, by (H2) we have
\begin{align}
\frac{1}{2R}\int_{|x|\leq R_{j}} (\partial_{r}(rV))_{-} |u|^{2} & \leq \frac{1}{2R}\int (\partial_{r}(rV))_{-}|\theta_{R_{j}}u|^{2}\label{pepe1}\\
& \leq \frac{A_{V}}{2R} \int_{|x|\leq R_{j}} |\D u|^{2} + \frac{C}{R_{j}^{2}}\int_{R_{j}\leq |x|\leq 2R_{j}} |u|^{2}\notag\\
& + \frac{A_{V}}{2R}\int_{R_{j}\leq |x|\leq 2R_{j}} |\D u|^{2}.\notag
\end{align}
Similarly, we obtain
\begin{align}
\frac{1}{R}\int_{|x|\leq R_{j}} |x||B_{\tau}||\D u||u| & \leq \left( \frac{1}{R}\int_{|x|\leq R_{j}} |\D u|^{2}\right)^{\frac{1}{2}}\left( \frac{1}{R}\int |\D(\theta_{R_{j}}u)|^{2}\right)^{\frac{1}{2}}\label{pepe2}\\
& \leq \frac{(A_{B} +\kappa)}{R}\int_{|x|\leq R_{j}} |\D u|^{2} + \frac{C}{R_{j}^{2}} \int_{R_{j}\leq |x|\leq 2R_{j}} |u|^{2}\notag\\
& +\frac{C}{R}\int_{R_{j}\leq |x|\leq 2R_{j}} |\D u|^{2},\notag
\end{align}
for any $\kappa >0$.

As a consequence, putting (\ref{pepe1}) and (\ref{pepe2}) into (\ref{pepe}), since $\{|x|\leq R\} \subset \{|x|\leq R_{j}\}$, $R_{j} \geq 1$, $A_{V}+2A_{B}<1$ and choosing $\kappa >0$ small enough, it may be concluded that
\begin{align}
\sup_{R\leq R_{j}} \frac{1}{R}\int_{|x|\leq R} \lambda|u|^{2} & \leq \frac{C}{R_{j}}\int_{R_{j}\leq |x|\leq 2R_{j}} (|\D u|^{2} + \lambda|u|^{2})\notag\\
& + C\int_{|x|=R_{j}} ((\lambda +V)|u|^{2} + |\D u|^{2}). \notag
\end{align}
Then taking the lim inf in $j$, by (\ref{Vuni}) and (\ref{limit12}) the theorem follows.
\end{proof}

If we are under the hypotheses of Theorem \ref{Theorem1} and $R(\lambda \pm i\varepsilon)f$ is solution of the equation (\ref{magneticequation5}), following again Ikebe and Saito \cite{IS} and Zubeldia \cite{Z} (section 2.5), we can prove the limiting absorption principle for the electromagnetic Helmholtz equation assuming sharp singularities on $V$ and $B_{\tau}$. Indeed, we can construct the unique solution $u_{\pm}$ of the equation (\ref{resjuan}) as the following limit in $(H^{1}_{A})_{loc}$
\begin{equation}
u_{\pm}(\lambda, f) = R(\lambda \pm i0)f = \lim_{\varepsilon \to 0} R(\lambda \pm i\varepsilon)f,
\end{equation}
where $u_{+}$ is called the outgoing solution, while $u_{-}$ is the incoming one.

\begin{thm}(LAP)\label{LAPchapter4}
Let $\lambda_{0}>0$. Under the hypotheses of Theorem \ref{Theorem1}, if moreover $V$ holds (\ref{Vuni}), then there exists a unique solution $u_{\pm}$ of the equation $(\ref{resjuan})$ such that for any $\lambda \geq \lambda_{0}>0$ satisfies the radiation condition
\begin{equation}\label{lap1}
\int |\D(e^{\mp i\lambda^{1/2}|x|}u_{\pm})|^{2} \leq C\int |x|^{2}|f|^{2}
\end{equation}
As a consequence
\begin{align}\label{lap2}
\int \frac{|u_{\pm}|^{2}}{|x|^{2}} \leq C\int |x|^{2}|f|^{2},
\end{align}
where $C>0$ is independent of $\lambda$.
\end{thm}

\begin{remark}\label{condicionuniqueness}
It is worth pointing out that the uniqueness result also follows by assuming some less restrictive conditions on the potentials (see also \cite{RT}). In fact, given $\lambda >0$ if we require that there exists $R_{0}=R_{0}(\lambda) >0$ such that
\begin{itemize}
\item [(\bf{H2a})] 
\begin{align}
\int_{|x|\leq R_{0}} (\partial_{r}(rV))_{-} |u|^{2} < A_{V} \Lambda_{R_{0}} + \int_{|x|\leq R_{0}} (\partial_{r}(rV))_{+} |u|^{2},\notag
\end{align}
\begin{align}
\left(\int_{|x|\leq R_{0}} |x|^{2}|B_{\tau}|^{2}|u|^{2}\right)^{1/2} < A_{B}\Lambda_{R_{0}}^{1/2},\notag
\end{align}
with
\begin{align}
\Lambda_{R_{0}} = \int_{|x|\leq R_{0}} |\nabla u|^{2} + \sup_{R\geq R_{0}}\frac{(d-1)}{2R}\int_{|x|=R} |u|^{2},\notag
\end{align}
\end{itemize}
where $A_{V} + 2A_{B} < 1$; 
\begin{itemize}
\item[(\bf{H2b})]
\begin{align}\notag
\frac{1}{R_{0}}\int_{|x|\geq R_{0}} \left[ (\partial_{r}V)_{-}\right] |u|^{2} &<  A_{V}''\lambda|||u|||_{R_{0}}^{2}+\int_{|x|\geq R_{0}} \frac{1}{|x|} (\partial_{r}V)_{+} |u|^{2}\notag
\end{align}
\begin{align}\notag
\frac{1}{R_{0}}\int_{|x|\geq R_{0}} |x|^{2}|B_{\tau}|^{2}|u|^{2} <  A''_{B}\lambda|||u|||_{R_{0}}^{2}.
\end{align}
\end{itemize}
where
\begin{equation}\label{H2}
1-A_V^{''}-2A_B^{''}>0, 
\end{equation} 
then Theorem \ref{unicidad3} follows. In this case condition (\ref{limit11}) can be replaced with a weaker one,
\begin{equation}
\lim\inf \int_{|x|=R} (\lambda |u|^{2} + |\D u|^{2}) \to 0 \quad \quad \textrm{as}\quad \quad R \to \infty.
\end{equation}
Furthermore, combining these assumptions with the condition (\ref{assV4}), it may be concluded the uniform Morrey-Campanato type estimate 
\begin{equation}\label{Morreyremark}
\lambda|||u|||_{R_{0}}^{2} + |||\D u|||_{R_{0}}^{2} \leq C(1+\varepsilon)(N_{R_{0}}(f))^{2},
\end{equation}
for all $\lambda \geq 0$ being $C$ independent of $\lambda$, $\varepsilon$ and the Sommerfeld radiation condition
\begin{equation}
\int_{|x|\geq 1} |\D(e^{\mp i\lambda^{1/2}|x|}u)|^{2} \leq C\int |x|^{2}|f|^{2} + (1+\varepsilon)(N_{R_{0}}(f))^{2},
\end{equation}
for the solution $u$ of the equation (\ref{magneticequation5}). 
\\
As a consequence, we deduce the limiting absorption principle for the equation (\ref{resjuan}) with potentials satisfying (H1), (H2a), (H2b) such that the solution holds the a-priori estimate (\ref{Morreyremark}). Note that the potentials given in examples \ref{ex} and \ref{exam} also hold the above assumptions. In fact, potentials that have the sharp singularity at the origin and that decay as a short-range potential at infinity are included. 
\end{remark}

\begin{remark}
In the same manner, under the assumptions of Theorem \ref{TheoremSRC}, we deduce that there exists a unique solution of the Helmholtz equation (\ref{resjuan}) such that for any $\lambda \geq \lambda_{0}$ satisfies
\begin{align}
& \lambda |||u|||_{1}^{2} \leq C (N_{1}(f))^{2},\notag
\end{align}
\begin{align}
& \sup_{R\geq 1}\, R\int_{|x|\geq R} |\nabla_{A}(e^{-i\sqrt{\lambda}|x|}u)|^{2} \leq C\left[\int |x|^{3}|f|^{2} + (N_{1}(f))^{2}\right],\notag
\end{align}
where $C=C(\lambda_{0})>0$.
\end{remark}

\subsection{Cross-section}\label{sectioncross}

Our next goal is to prove existence and uniqueness of the cross-section of the solution $u$ of the equation (\ref{resjuan}).

Let us denote $r=|x|$, $\omega = \frac{x}{|x|}$ and $S^{d-1} = \{x\in \Rd : |x|=1\}$ and
\begin{equation}
(\mathcal{F}(\lambda, r)f)(\omega) = \lambda^{1/2}r^{\frac{d-1}{2}}e^{-i\lambda^{1/2}r}u(r\omega), \quad \quad \omega \in S^{d-1}, 
\end{equation}
where $u = R(\lambda + i0)f$ is the outgoing solution of the equation (\ref{resjuan}). Then the cross-section is given by 
\begin{equation}\label{limitcross}
\mathcal{G}_{\lambda}(\omega) = \lambda^{-\frac{(d-1)}{2}} \lim_{r\to \infty} \left|(\mathcal{F}(\lambda, r)f)(\omega)\right|^{2} 
\end{equation}
provided that the limit exists.

We start using the resolvent estimate (\ref{BP}) together with some modifications of the ideas of Isozaki \cite{Is} and Iwatsuka \cite{Iw} to show that the cross-section is a positive Radon measure on $S^{d-1}$. 

To this end, we first need to state a property related to the surface integral involving the solution $u$ of the equation (\ref{resjuan}). Let us denote
\begin{equation}
\Dr = \D^{r}  - i\lambda^{1/2} + \frac{(d-1)}{2r}.
\end{equation}
Note that under the hypotheses of Theorem \ref{LAPchapter4} from (\ref{lap1}) and (\ref{lap2}) it follows that
\begin{equation}
\int |\Dr u|^{2} \leq C\int |x|^{2}|f|^{2},
\end{equation}
for some $C > 0$.

Thus we state the following result due to Iwatsuka. Unless otherwise stated, in this paragraph we will work under the assumptions of Theorem \ref{LAPchapter4} and we will take $u = R(\lambda + i0)f$.

\begin{lem}[Proposition 3.4, \cite{Iw}]\label{lema2.5}
Let $f$ such that $\Vert |\cdot |f \Vert_{L^{2}}$ is finite, $u= R(\lambda + i0)f$. Let $v$ such that $\int \frac{|v|^{2}}{|x|^{2}} < \infty$ and $\int \left|\D(e^{-i\lambda^{1/2}|x|}v)  \right|^{2} < +\infty$. Then
\begin{equation}\label{eqlema2.5}
\int_{|x|=r} (\Dr u)\bar{v}d\sigma_{r} \quad \to \quad 0 \quad \quad \textrm{as} \quad \quad r\to \infty.
\end{equation}
\end{lem}

\vspace{0.2cm}
Our next result is a fundamental step to prove the existence of $\mathcal {G}_{\lambda}$.

\begin{pro}\label{lema2.8}
Let $f$ such that $\Vert |  \cdot |f \Vert_{L^{2}} < \infty$. Given $\phi \in C^{\infty}(S^{d-1})$ there exists the limit
\begin{equation}\label{weaklimit}
\lim_{r\to\infty} \int_{|x|=1} \left|\mathcal{F}(\lambda,r)f(\omega)\right|^{2}\phi(\omega).
\end{equation}
\end{pro}

\begin{proof}
Let $\phi(\omega) \in C^{\infty}(S^{d-1})$ and we define
\begin{equation}\label{v}
v = \rho(r)u(r\omega)\phi(\omega) \quad \quad \left(r=|x|, \omega= \frac{x}{r}\right),
\end{equation}
where $\rho(r)\in C^{\infty}(\R^{+})$ such that $\rho(r)=0$ if $r<1$ and $\rho(r)=1$ if $r>2$. Note that we have $\int \frac{|v|^{2}}{|x|^{2}} <\infty$ and $\int |\D(e^{-i\lambda^{1/2}|x|}v)|^{2} < \infty$. 

Let
\begin{equation}\label{g}
g = (\D^{2} + V + \lambda)v.
\end{equation}  
Then, by a straightforward calculation we have 
\begin{align}\label{expresiong}
g& = f\rho\phi + \frac{2\rho}{|x|}\D^{\bot} u \cdot \nabla_{\omega}\phi + \frac{\rho}{|x|^{2}}\Lambda \phi u + 2\rho' \D^{r}u\phi + \rho'' u\phi + \frac{\rho'(d-1)}{|x|}u\phi
\end{align}
where $\Lambda$ is the Laplace Beltrami operator on $S^{d-1}$. Therefore, from the a-priori estimates (\ref{BP}) and (\ref{Morrey}) it follows that
\begin{align}
\int |x|^{2}|g|^{2} &\leq C\left[\int |x|^{2}|f|^{2} +\int |\D^{\bot}u|^{2} +  \int |\D^{r}u - i\lambda^{1/2}u|^{2} +(1+\lambda)\int \frac{|u|^{2}}{|x|^{2}}\right]\notag\\
&\leq C_{\lambda}\left( \int |x|^{2}|f|^{2} \right).\notag
\end{align}

Now letting $u=R(\lambda + i0)f$, by integration by parts we see that
\begin{align}\label{(2.8)}
\int_{|x|<r} \left\{(\D^{2}v)\bar{u} - v(\overline{\D^{2}u})\right\}  = \int_{|x|=r} \{(\Dr v)\bar{u} - v(\overline{\Dr u})\} + 2i\sqrt{\lambda}\int_{|x|=r}v\bar{u}.
\end{align}
Note that letting $r \to \infty$, the left hand side of the above identity tends to
$$
(g, R(\lambda+i0)f) - (v, f),
$$
which is finite as it is bounded by $C \int |x|^{2}(|f|^{2} + |g|^{2})$.
Since $\Dr v = (\Dr u \phi)\rho$ if $r>2$, by Lemma \ref{lema2.5} the first term of the right hand side of (\ref{(2.8)}) tends to $0$ as $r\to \infty$. Moreover, by the definitions of $v$ and $\mathcal{F}(\lambda,r)f$ it follows that

\begin{align}
\lim_{r\to \infty} 2i\sqrt{\lambda}\int_{|x|=r} v\bar{u} &= \lim_{r\to \infty} 2i\sqrt{\lambda}\int_{|x|=r} |u|^{2}\phi\notag\\
&= 2i\lambda^{-\frac{1}{2}}\lim_{r\to \infty}\int_{|x|=1} |\mathcal{F}(\lambda,r)f|^{2}\phi.\notag
\end{align}
Therefore, we obtain

\begin{equation}\label{cong}
(g,R(\lambda + i0)f) - (v, f) = 2i\lambda^{-\frac{1}{2}}\lim_{r\to \infty} \int_{|x|=1} |\mathcal{F}(\lambda,r)f|^{2}\phi,
\end{equation}
which proves the proposition.
\end{proof}

The existence of the cross-section is established by our next theorem.

\begin{thm}\label{azkenx}
Under the hypotheses of Theorem \ref{LAPchapter4}, let $u = R(\lambda + i0)f$. Then, there exists a positive Radon measure $\mu_\lambda$ on $S^{d-1}$ such that
for each $\phi \in C(S^{d-1})$, 
\begin{equation}\label{mu}
\int_{|x|=1} \phi(\omega) d\mu_\lambda
 = \lim_{r\to \infty} \int_{|x|=1} |\mathcal{F}(\lambda, r)f(\omega)|^{2} \phi(\omega) d\sigma(\omega).
\end{equation}
Moreover for some $C>0$ and $p_0$ with $\frac{1}{p_{0}} + \frac{1}{2} - \frac{1}{d} = 1$, 
\begin{equation}\label{xxx}
\mu_\lambda(S^{d-1}) = \lambda^{1/2} \Im \int f\bar{u}\leq C \min\bigl(N_{1}(f))^{2},\lambda^{1/2}\Vert f \Vert_{L^{p_{0}}} \Vert |x|f \Vert_{L^{2}}\bigr).
\end{equation}

\end{thm}

\begin{proof}

Firstly, observe that by definition of $\mathcal{F}(\lambda, r)f$, we have
\begin{equation}
\lambda^{-1/2}\int_{|x|=1} |\mathcal{F}(\lambda, r)f(\omega)|^{2} d\sigma(\omega) = \lambda^{1/2} \int_{|x|=r} |u(r\omega)|^{2} d\sigma_{r}.
\end{equation}
Secondly, if we multiply equation (\ref{resjuan}) by $\bar{u}$, integrate over the ball $\{|x| \leq r\}$ for some $r>0$ and take the imaginary part, we obtain
\begin{align}
\Im \int_{|x|=r} \frac{x}{|x|} \cdot (\D u - i\lambda^{1/2}\frac{x}{|x|}u) \bar{u} + \lambda^{1/2}\int_{|x|=r} |u|^{2} & = \Im \int_{|x|\leq r} f\bar{u}.
\end{align}
Computing the $lim$ $inf$ as $r\to \infty$ on the both sides of the above identity, by (\ref{lap1}) and (\ref{lap2}) it follows that
\begin{equation}
\lambda^{1/2} \lim_{r\to \infty} \inf \int_{|x|=r} |u|^{2} = \Im \int_{\Rd} f\bar{u}.
\end{equation}
As a consequence, we get
\begin{equation}\label{limphi}
\lim_{r\to \infty} \inf \int_{|x|=1} |\mathcal{F}(\lambda,r)f(\omega)|^{2} = \lambda^{1/2}\Im \int f\bar{u}.
\end{equation}
In addition, by the existence of the limit (\ref{weaklimit}), taking $\phi(\omega) = 1$, it may be concluded that
\begin{equation}
\lim_{r\to \infty} \int_{|x|=1} |\mathcal{F}(\lambda,r)f(\omega)|^{2} = \lambda^{1/2} \Im \int f\bar{u}.
\end{equation}

For each $\phi \in C^{\infty}(S^{d-1})$, let us define $T: C^{\infty}(S^{d-1}) \to \mathbb{C}$ such that
$$
T\phi = \lim_{r\to \infty} \int_{|x|=1} |\mathcal{F}(\lambda, r)f(\omega)|^{2} \phi(\omega) d\sigma(\omega).
$$
Note that since $\Lambda=\lambda^{1/2} \Im \int f\bar{u}$ is uniformly bounded, it follows that
$$
|T\phi| \leq \Lambda \Vert \phi \Vert_{L^{\infty}}.
$$ 
In addition, since $C^{\infty}(S^{d-1})$ is dense in $C(S^{d-1})$, we can extend the operator $T$ for any $\phi \in C(S^{d-1})$. Thus by the Riesz representation theorem we conclude that there exists a measure $\mu_\lambda$ on $S^{d-1}$ such that
$$
T\phi = \int_{|x|=1} \phi(\omega) d\mu_\lambda.
$$
Thus we get (\ref{mu}). In particular,
$$
\mu_\lambda(S^{d-1}) = \int_{|x|=1} d\mu_\lambda= \lim_{r \to \infty} \int_{|x|=1} |\mathcal{F}(\lambda, r)f(\omega)|^{2}.
$$
Finally estimate \eqref{xxx} follows from \eqref{Morrey}. 
\end{proof}

In our next result we use the sharp Sommerfeld condition (\ref{keySRC}) and the a-priori estimate (\ref{Morrey}) to prove that $\mu_\lambda$ is an absolutely continuous measure. 
We emphasize that for this purpose, we need more restrictions on the potentials and we will work under the hypotheses of Theorem \ref{TheoremSRC}.

\begin{thm}\label{cross-section}
Under the hypotheses of Theorem \ref{TheoremSRC}, there exists  a function $\mathcal{G}_{\lambda}$ in $L^{1}(S^{d-1})$ such that 
$$\lim_{r\to\infty}|\F(\lambda, r)f|^{2}=\mathcal{G}_{\lambda} \quad \quad \text{in}  \quad \quad  L^{1}(S^{d-1}).$$
Moreover, for some $C>0$ depending on $\lambda_0$ and $p_0$ with $\frac{1}{p_{0}} + \frac{1}{2} - \frac{1}{d} = 1$, yields
\begin{equation}\label{gestimate}
\int_{S^{d-1}} \mathcal{G}_{\lambda}(\omega) d\sigma(\omega)=\lambda^{1/2} \Im \int f\bar{u} \leq C \min\bigl(N_{1}(f))^{2},\lambda^{1/2}\Vert f \Vert_{L^{p_{0}}} \Vert |\cdot|f \Vert_{L^{2}}\bigr).
\end{equation}
\end{thm}

\begin{proof}
Our proof starts with the observation that from (\ref{keySRC}) and (\ref{Morrey}) there exists a sequence $\{r_{n}\}_{n\in\mathbb{N}}$ that tends to infinity such that for each $r_{n}$
\begin{equation}\label{ffp1}
r_{n}^{2} \int_{|x|=r_{n}} \left|\nabla_{A} u - i\lambda^{1/2}\frac{x}{|x|}u\right|^{2}d\sigma_{r} < +\infty.
\end{equation}
and
\begin{equation}\label{ffp2}
\lambda\int_{|x|=r_{n}} |u|^{2}d\sigma_{r}< +\infty,
\end{equation}
respectively.

We write $x=r_{n}\omega$, where $\omega\in S^{d-1}$ and we take
\begin{equation}
h_{n}(\omega) = \left(\F(\lambda, r_{n})f\right)(\omega).
\end{equation}
Thus we have
\begin{align}
\int_{S^{d-1}}|h_{n}(\omega)|^{2} d\sigma(\omega) & = \lambda r_{n}^{d-1}\int_{S^{d-1}} |u(r_{n}\omega)|^{2}d\sigma(\omega)\notag\\
& = \lambda \int_{|x|=r_{n}} |u(x)|^{2} d\sigma_{r}.\notag
\end{align}
Furthermore, by the diamagnetic inequality (\ref{diamagnetic}) we get
\begin{align}
\int_{S^{d-1}} |\nabla_{\omega} |h_{n}(\omega)||^{2}d\sigma(\omega) & = \lambda r_{n}^{d+1}\int_{S^{d-1}} \left|\nabla^{\tau}\left|e^{-i\lambda^{1/2}r_{n}}u(r_{n}\omega)\right|\right|^{2}d\sigma(\omega)\notag\\
& \leq \lambda r_{n}^{2}\int_{|x|=r_{n}} \left|\nabla\left| e^{-i\lambda^{1/2}r_{n}}u\right|\right|^{2} d\sigma_{r}\notag\\
& \leq \lambda r_{n}^{2}\int_{|x|=r_{n}}  \left|\nabla_{A}\left( e^{-i\lambda^{1/2}r_{n}}u\right)\right|^{2} d\sigma_{r}.\notag
\end{align}
Hence it may be concluded that $|h_{n}| \in H^{1}(S^{d-1})$. Now, by the Rellich theorem we deduce that $\exists  g_{\lambda} \in L^{2}(S^{d-1})$ such that $|h_{n}| \to g_{\lambda}$ in $L^{2}(S^{d-1})$ and therefore 
that $|h_{n}| ^2\to g_{\lambda}^2$ in $L^{1}(S^{d-1})$. From Theorem \ref{azkenx} we conclude that the limit does not depend on the sequence $r_n$
and that  $\mathcal{G}_{\lambda} = g_{\lambda}^{2}$ . 

Let us prove now (\ref{gestimate}). By Theorem \ref{azkenx} again we obtain
\begin{align}\label{gg5}
\int_{|x|=1}  g^2_{\lambda}(\omega) d\sigma(\omega) = \int_{|x|=1} \mathcal{G}_{\lambda}(\omega) d\sigma(\omega) = \lambda^{1/2}\Im\int_{\Rd} f\bar{u},
\end{align}
which combining with (\ref{Morrey}) gives the desired estimate and the proof is complete.
\end{proof}

\begin{remark}\label{constfar}
Observe that in the case that $A \equiv 0$, then $h_{n} \in S^{d-1}$ and by the same arguments, we deduce that there exists $g_{\lambda} \in L^{2}(S^{d-1})$ such that $h_{n} \to g_{\lambda}$ in $L^{2}(S^{d-1})$. 
\end{remark}

\subsection{Spectral representation}\label{sectionspectral}

We will prove some spectral properties of the magnetic Schr\"odinger operator 
$$-H_{A} = -(\D^{2} + V)$$
 in $\Rd$. 

Let $E(B)$ the spectral measure associated with $H_{A}$, where $B$ varies over all Borel sets of the reals. From Theorem \ref{LAPchapter4} it follows that for any $f$ such that $\Vert (1+|\cdot|)f \Vert_{L^{2}}<\infty$ a unique solution $u_{\pm}(\lambda, f)$ of the equation 
$$
(H_{A}+\lambda)u_{\pm} = f
$$
satisfying the corresponding Sommerfeld radiation condition 
$$\int |\D(e^{\mp i\lambda^{1/2}|x|}u_{\pm})|^{2} < \infty$$ 
can be constructed as the limit
\begin{equation}\label{lll}
u_{\pm}(\lambda, f) = \lim_{\varepsilon \to 0}\, R(\lambda\pm i\varepsilon)f \quad \quad \text{in} \quad \quad (H^{1}_{A})_{loc},
\end{equation}
where $u(\lambda \pm i\varepsilon) = R(\lambda \pm i\varepsilon)$. Let $\Delta=(\lambda_{1}, \lambda_{2})$ where $0<\lambda_{1} < \lambda_{2} < \infty$. Then, employing the following well-known Stone's formula (see \cite{DS})
\begin{equation}\notag
(E(\Delta)f, f) =\lim_{\varepsilon \to 0} \lim_{\nu\to 0} \frac{1}{2\pi i}\int_{\lambda_{1}-\nu}^{\lambda_{2}+\nu} (R(\lambda - i\varepsilon)f - R(\lambda + i\varepsilon)f, f)\, d\lambda \quad \quad (f\in L^{2}),
\end{equation}
by (\ref{lll}), the fact that $(R(\lambda - i\varepsilon)f - R(\lambda + i\varepsilon)f, f)$ is uniformly bounded for $(\lambda, \varepsilon) \in [\lambda_{1}, \lambda_{2}] \times [0,1]$, together with the Lebesgue dominated convergence theorem, one obtains
\begin{equation}
(E(\Delta)f, f) = \frac{1}{2\pi i} \int_{\Delta} (u_{-}(\lambda,f) - u_{+}(\lambda,f), f)\, d\lambda.
\end{equation}

Let us show the continuity of $(u_{-}(\lambda,f) - u_{+}(\lambda,f), f)$ for $f$ with compact support. Let $\Omega \subset \Rd$ and $supp \, f \subset \Omega$. Let us define $\eta$ such that $\eta(x)=1$ if $x\in \Omega$ and $\eta(x)=0$ when $x \notin \Omega$. Note that $(u_{-}(\lambda,f) - u_{+}(\lambda,f), f) = 2i\Im \int f \bar{u_{\lambda}}$, where $u_{\lambda}=R(\lambda + i0)f$. Let us denote $w= u_{\lambda} - u_{\mu}$ and take $v=\eta w$. Then it follows that
\begin{align}\notag
H_{A}w + \lambda w = (\lambda-\mu)\eta u_{\mu} -2\nabla \eta \cdot \D w - \Delta \eta w.
\end{align}
Thus, we have
\begin{align}
\Im \int f(\bar{u}_{\lambda} - \bar{u}_{\mu})  \leq C\left( \frac{(\lambda - \mu)^{1/2}}{\lambda^{1/2}} + \int_{\Omega} (|\D w|^{2} +|w|^{2}) \right)N(f).\notag
\end{align}
From Rellich's theorem we obtain $\Im \int f (\bar{u}_{\lambda} - \bar{u}_{\mu}) \to 0$ as $\lambda \to \mu$.
Using a simple density argument  and  Theorem \ref{Theorem1} and \eqref{xxx} we have proved the following result.

\begin{thm}
Under the hypotheses of Theorem \ref{LAPchapter4}, 
\begin{equation}
(f, f) = \frac{1}{2\pi i} \int_{0}^{\infty} (u_{-}(\lambda,f) - u_{+}(\lambda,f), f)\, d\lambda= \frac{1}{\pi } \int_{0}^{\infty} \lambda^{-1/2}\mu_\lambda (S^{d-1}) \,d\lambda.
\end{equation}
\end{thm}

If instead of Theorem \ref{LAPchapter4} and \eqref{xxx} we use Theorem \ref{TheoremSRC} and \eqref{gestimate} we obtain the following.

\begin{thm}
Under the hypotheses of Theorem \ref{TheoremSRC}, let $E(0,\infty)$ be the projection onto the positive part of the spectrum of $-H_{A}$. Then 
\begin{align}
(E(0,\infty) f, f) &=\frac{1}{2\pi i} \int_{0}^{\infty} (u_{-}(\lambda,f) - u_{+}(\lambda,f), f)\, d\lambda \notag\\
&=\frac{1}{\pi } \int_{0}^{\infty}\lambda^{-1/2}\Vert \mathcal G_\lambda(\omega)\Vert_{L^{1}(S^{d-1})}\, d\lambda.\notag
\end{align}

\end{thm}

\section{Weighted inequalities and applications}\label{section6}

The last section provides some new estimates which are a consequence of the main results of this work. On the one hand, by the uniform resolvent estimate (\ref{alpha0}) we deal with a weighted Sobolev inequality that allows us to solve the corresponding evolution problem. On the other hand, from the sharp Sommerfeld condition (\ref{keySRC}) and the a-priori estimate (\ref{Morrey}) we give a new estimate for $|\nabla |u|^2|.$
\subsection{Weighted inequalities}\label{sectionevo}

In the sequel, we will study the evolution problem associated to the magnetic Schr\"odinger operator $H_{A} = \D^{2} + V$. 

Let us define a Sobolev weight $\omega$ as a positive function $\omega: \Rd \to [0, \infty)$ such that
\begin{equation}\label{sobolevweight}
\int |g|^{2} \omega \leq c(\omega) \int |\nabla g|^{2}.
\end{equation}

If $\omega$ belongs to the Morrey-Campanato class $\mathcal{L}^{2,p}(\Rd)$ with $p>1$, then $|\omega|$ is a Sobolev weight. Here, the Morrey-Campanato class $\mathcal{L}^{\alpha, p}(\Rd)$ for $\alpha >0$ and $1\leq p\leq d/\alpha$ is defined by
$$
{\tiny \mathcal{L}^{\alpha, p}(\Rd) = \left\{ V\in L^{p}_{loc}(\Rd) : \sup_{r>0} \left(\frac{1}{r^{d-p\alpha}}\int_{|x|<r} |V(x)|^{p} dx \right)^{1/p} <\infty \right\}.}
$$
Thus
$$
\Vert V \Vert_{\mathcal{L}^{\alpha, p}(\Rd)} : =  \sup_{r>0} \left(\frac{1}{r^{d-p\alpha}}\int_{|x|<r} |V(x)|^{p} dx \right)^{1/p}.
$$
In particular, $\omega(x)=\frac{1}{|x|^{2}} \in \mathcal{L}^{2, p}(\Rd)$ with $p<d/2$.

We begin by proving a weighted estimate for the stationary case. 

\begin{thm}
Under the hypotheses of Theorem \ref{Theorem1}, let $u$ be a solution of the equation
\begin{equation}
H_{A} u + (\lambda \pm i\varepsilon) u = f
\end{equation}
and $\omega$ a Sobolev weight. Then there exists a constant $C = C(\omega)$ independent of $\varepsilon$ and $\lambda$ such that the following a-priori estimate holds
\begin{equation}\label{w1}
\int \left| R(\lambda \pm i\varepsilon)f \right|^{2} \frac{\omega^{1/2}}{|x|} \leq C \int |f|^{2} \frac{|x|}{\omega^{1/2}}.
\end{equation}
\end{thm}

\begin{proof}
By Theorem \ref{Theorem1}, if $0<\varepsilon < \lambda$ we have
\begin{equation}\label{33}
\int \left| \D(e^{\mp i \lambda^{1/2} |x|}R(\lambda \pm i\varepsilon))f )\right|^{2} \leq C \int |x|^{2}|f|^{2}
\end{equation}
and for $\lambda \leq \varepsilon$,
\begin{equation}\label{44}
\int \left| \D(R(\lambda \pm i\varepsilon))f \right|^{2} \leq C \int |x|^{2}|f|^{2},
\end{equation}
where $C$ is independent of $\varepsilon$ and $\lambda$. Since $\omega$ is a Sobolev weight, by the diamagnetic inequality (\ref{diamagnetic}), (\ref{33}) and (\ref{44}), yields
\begin{equation}\label{55}
\int \left| R(\lambda \pm i\varepsilon)f \right|^{2} \omega \leq Cc(\omega)\int \left| \D(R(\lambda \pm i\varepsilon))f \right|^{2} \leq C(\omega) \int |x|^{2}|f|^{2},
\end{equation}
for any $\lambda \in \R$, $\varepsilon >0$. By duality, we obtain
\begin{equation}\label{dual}
\int \frac{\left|R(\lambda \pm i\varepsilon)f  \right|^{2}}{|x|^{2}} \leq C  \int |f|^{2}\omega^{-1}.
\end{equation}
Interpolating (\ref{55}) and (\ref{dual}) follows (\ref{w1}) and the proof is complete.
\end{proof}

\begin{remark}
Estimate (\ref{w1}) implies that the operator $Tf(x) = \frac{\omega^{1/4}}{|x|^{1/2}}f(x)$ is H-supersmooth, which means (see \cite{KaYa}) that for all $f \in D(T) \subset L^{2}(\Rd)$, $\lambda \in \R$ and $\varepsilon >0$, there exists a positive constant independent of $\lambda$ and $\varepsilon$ such that
\begin{equation}
\left|\left(R(\lambda \pm i\varepsilon)Tf, Tf\right)\right| \leq C\Vert f \Vert_{L^{2}(\Rd)},
\end{equation}
where $R(z) = (H_{A} + z)^{-1}$ and $(\cdot, \cdot)$ denotes the inner product in $L^{2}(\Rd)$.

In particular, $T$ is $H$-smooth. Indeed, it satisfies
\begin{equation}
\left|\Im \left(R(\lambda \pm i\varepsilon)Tf, Tf\right)\right| \leq C\Vert f \Vert_{L^{2}(\Rd)},
\end{equation}
which is equivalent to 
\begin{equation}\label{weighted}
\int_{-\infty}^{+\infty} \int_{\Rd} \left| e^{i H_{A}t}f(x) \right|^{2} \frac{\omega^{1/2}}{|x|} dx dt \leq C(\omega) \int_{\Rd} |f(x)|^{2}dx.
\end{equation}
We refer the reader to \cite{KaYa} for more details.
\end{remark}

\begin{remark}
In the case that $A \equiv 0 \equiv V$, Ruiz and Vega \cite{RV1} show the smoothing estimate
\begin{equation}
\left| \Im\left(R_{0}(\lambda \pm i\varepsilon)\omega_{1}f, \omega_{1}f \right) \right| \leq C \Vert f \Vert_{L^{2}},
\end{equation}
where $R_{0}(z) = (\Delta + z)^{-1}$, for any $\omega_{1} \in \mathcal{L}^{2,q}$ with $q> \frac{d-1}{2}$. As a consequence, it may be concluded the following evolution inequality
\begin{equation}
\int_{-\infty}^{+\infty} \int_{\Rd} \left|e^{it\Delta}f(x) \right|^{2} \omega_{1} dx dt \leq C(\omega_{1}) \int_{\Rd} |f(x)|^{2}dx,
\end{equation}
where $C(\omega_{1})= C c(\omega_{1})$ with $c(\omega_{1}) = \Vert \omega_{1} \Vert_{\mathcal{L}^{2,q}}$.

Note that there exist Sobolev weights $\omega$ such that if $\omega_{1} = \frac{\omega^{1/2}}{|x|}$, then 
\begin{equation}\label{condomega}
\omega \in \mathcal{L}^{2,p} \quad \quad  \text{and} \quad \quad \omega_{1} \notin \mathcal{L}^{2,q} 
\end{equation}
with $p>1$ and $q>\frac{d-1}{2}$, respectively. In order to see that, for any $R>0$ we define
$$
\omega_{0}(x) = \chi_{[0,1]^{d-k}}(x_{1})\chi_{[R, 2R]^{k}}(x_{2}),
$$
where $x=(x_{1}, x_{2})$ with $x_{1}\in \R^{k}$, $x_{2}\in \R^{d-k}$, $1\leq k < d$. Take $\omega(x) = \frac{1}{R^{d\left(1-\frac{1}{p}\right)}}\omega_{0}(x)$. Then a computation shows that for $k=d-2$
$$
\Vert \omega \Vert_{\mathcal{L}^{2,p}} \leq C \quad \quad \text{and} \quad \quad \Vert \omega_{1} \Vert_{\mathcal{L}^{2,q}} \geq C R^{1- \frac{2}{q} -d\left( 1-\frac{1}{p} \right)},
$$
where $C>0$ is independent of $R$. Thus, given $q > \frac{d-1}{2}$ and choosing $p>1$ such that
$$
1 - \frac{2}{q} -d \left(1-\frac{1}{p} \right) >0,
$$
which is true for any $d\geq 5$, then we deduce (\ref{condomega}). The proof can be adapted for $d\geq 4$ \cite{Lem}. In that case, $k$ needs to be fractional. For more details and examples in this direction we refer the reader to \cite{BBCRV}.
\end{remark}

We may now state the result that gives weighted estimates for the evolution problem.

\begin{thm}
Let us consider the initial value problem

\begin{equation}\label{initialvalue}
\left\{ \begin{array}{ll}
( i \partial_{t} + H_{A}) u(t,x) = F(x, t) &(x, t) \in \Rd \times \R \\
u(x,0) = u_{0}(x),
\end{array} \right.
\end{equation}
with $F \in L^{2}\left(\frac{\omega^{1/4}}{|x|^{1/2}}dxdt\right)$, $u_{0} \in L^{2}(\Rd)$. Let $\omega$ be a Sobolev weight. Then, under the hypotheses of Theorem \ref{Theorem1}, the solution $u(x, t)$ satisfies
\begin{align}\label{weightedF}
\int_{-\infty}^{+\infty} \int_{\Rd} \left| e^{i H_{A}t}u_{0}(x) \right|^{2} \frac{\omega^{1/2}}{|x|} dx dt &\leq C(\omega) \int_{\Rd} |u_{0}(x)|^{2}dx\\
& + C(\omega)\int_{-\infty}^{+\infty} \int_{\Rd} |F(x,t)|^{2} \frac{|x|}{\omega^{1/2}}dxdt.\notag
\end{align}
\end{thm}

\begin{proof}
When $F = 0$, (\ref{weightedF}) follows from (\ref{weighted}). Note that here we only need the estimate for the imaginary part of the resolvent.

Let us assume that $u_{0}=0$. In that case, we can write the solution as follows
\begin{equation}
u(x, t) = v(x, t) + r(x, t)
\end{equation}
with 
\begin{equation}
v(x, t) = \int_{-\infty}^{+\infty} e^{it\tau} (H_{A} - \tau + i0)^{-1} F(x, \hat{\tau}) d\tau,
\end{equation}
where $F(x, \hat{\cdot})$ denotes the Fourier transform in the $t$-variable and $r(x,t)$ is the correction term which aligns the solution with the initial datum and is defined by
\begin{equation}
r(x, t) = c S S^{*} \tilde{F}.
\end{equation}
Here $S: L^{2}(\Rd) \to L^{2}\left(\frac{\omega^{1/4}}{|x|^{1/2}}dxdt \right)$ is given by $Sf = e^{itH_{A}}f$, $S^{*}$ is its adjoint, $\tilde{F}(x, t) = sign \, t\, F(x,t)$ and $c$ is an appropriate constant. Estimate (\ref{weightedF}) for $u_{0}=0$ follows by standard techniques that can be found in \cite{RV1}, \cite{RV2} or \cite{BBR}.  
\end{proof}

\subsection{An inequality for $|\nabla |u|^{2}|$}

In what follows, let $R_{0}>0$ and assume that $supp \, f \subset B(0, R_{0}) : = \{x\in \Rd : |x| \leq R_{0}\}$. 
Let us now make the following observation. Let $\lambda_{0}, \varepsilon >0$ with $\lambda_{0} > 0$. Then, under the hypotheses of Theorem \ref{TheoremSRC}, by Proposition \ref{som}, together with the a-priori estimate (\ref{Morrey}), it may be concluded that there exists a unique solution $u$ of the equation (\ref{resjuan}) such that for any $R \geq 1$ satisfies
\begin{equation}\label{apri}
\lambda|||u|||_{1}^{2} \leq C \min\{R_{0}, 1\}\int |f|^{2}
\end{equation}
and
\begin{align}\label{gradi}
R \int_{|x| \geq R} |\D(e^{-i\lambda^{1/2}|x|}u)|^{2} & \leq C\left[R_{0}^{3} + \min\{R_{0}, 1\} \right]\int |f|^{2},
\end{align} 
where $C=C(\lambda_{0})>0$.

We are now in a position to show the main result of this section.

\begin{thm}\label{theoremgradient}
Let $\lambda_{0}>0$, $R_{0} > 0$ and $supp \, f \subset B(0, R_{0})$. Under the hypotheses of Theorem \ref{TheoremSRC}, the solution of the Helmholtz equation (\ref{resjuan}) satisfies for any $\lambda \geq \lambda_{0}$
\begin{align}
\int_{R \leq |x| \leq 2R}  \left| \nabla |u|^{2}\right|  \leq C \left(R_{0}^{3} + \min\{R_{0}, 1\}  \right)^{\frac{1}{2}}\left(\frac{\min\{R_{0}, 1\} }{\lambda} \right)^{\frac{1}{2}} \int |f|^{2} 
\end{align}
if $R \geq 1$, and 
\begin{align}
\int_{|x|\leq 1} \left| \nabla |u|^{2} \right| \leq \frac{C}{\lambda^{1/2}} \min\{R_{0}, 1\}  \int |f|^{2},
\end{align}
where $C=C(\lambda_{0}) >0$.
\end{thm}

\begin{proof}
Let $R \geq 1$. Let $|x| \sim R$ denote the set $\{x\in \Rd : R \leq |x| \leq 2R\}$. By the diamagnetic inequality (\ref{diamagnetic}), (\ref{apri}) and (\ref{gradi}), we have
\begin{align}
\int_{|x|\sim R} \left| \nabla |u|^{2} \right| & \leq \left( R\int_{|x|\sim R} |\D(e^{-i\lambda^{1/2}|x|}u)|^{2}\right)^{\frac{1}{2}}\left(\frac{1}{R}\int_{|x|\sim R} |u|^{2} \right)^{\frac{1}{2}}\notag\\
& \leq C\left( R_{0}^{3} + \min\{R_{0}, 1\} \right)^{\frac{1}{2}} \left(\frac{\min\{R_{0}, 1\} }{\lambda} \right)^{\frac{1}{2}}\int |f|^{2}.\notag
\end{align}
In addition, by the diamagnetic inequality (\ref{diamagnetic}) and (\ref{apri}) it follows that
\begin{align}
\int_{|x|\leq 1} \left| \nabla |u|^{2} \right| & \leq 2\int_{|x|\leq 1} |\D u| |u|\notag\\
& \leq 2|||\D u|||_{1} |||u|||_{1} \notag\\
& \leq \frac{C}{\lambda^{1/2}} \min\{R_{0}, 1\} \int |f|^{2}\notag,
\end{align}
and the proof is complete.
\end{proof}

As a consequence, by the Sobolev embedding $W^{1,1}(\Rd) \subset L^{\frac{d}{d-1}}(\Rd)$, we can estimate $\left( \int_{|x|\sim R} |u|^{\frac{2d}{d-1}} \right)^{\frac{d-1}{d}}$ for any $R\geq 1$. To this end, let us define a smooth function $\eta$ such that 
\begin{equation}\notag
\eta(r) = \left\{ \begin{array}{ll}
1 & \textrm{if $R \leq r \leq 2R$},\\
0 & \textrm{if $r < R -1$, $r > 2R +1$}.
\end{array} \right.
\end{equation}
Thus, we deduce that
\begin{align}\notag
\Vert |u|^{2} \chi_{\{|x|\sim R\}}\Vert_{L^{\frac{d}{d-1}}} &\leq \left( \int_{|x|\sim R} \left(\eta(|x|)|u|^{2}\right)^{\frac{d}{d-1}} \right)^{\frac{d-1}{d}}\notag\\
& \leq \frac{1}{R}\int_{|x|\sim R} |u|^{2} + \int_{|x|\sim R} \nabla |u|^{2}\notag\\
& \leq C \left[ \frac{R_{0}}{\lambda} + \left( R_{0}^{3} +\min\{R_{0}, 1\} \right)^{\frac{1}{2}}\frac{\min\{R_{0}, 1\} ^{1/2}}{\lambda^{1/2}}\right] \int |f|^{2}.\notag
\end{align}

\begin{remark}
Note that in the particular case of $R_{0} = 1$ and $\lambda > > 1$, yields
\begin{equation}\notag
 \int_{|x| \sim R} \left|\nabla |u|^{2} \right| \leq \frac{C}{\lambda^{1/2}} \int |f|^{2}
\end{equation}
and
\begin{equation}\notag
\left( \int_{|x|\sim R} |u|^{\frac{2d}{d-1}} \right)^{\frac{d-1}{d}} \leq \frac{C}{\lambda^{1/2}} \int |f|^{2}.
\end{equation}
\end{remark}

\end{document}